\documentclass[a4paper,twoside,10pt]{article}
\usepackage[left=3.0cm,right=3.0cm,top=3.5cm,bottom=3.5cm]{geometry}
\usepackage{mathrsfs}
\usepackage{mathtools}
\usepackage{graphicx}
\usepackage{epstopdf} 
\usepackage{stmaryrd}
\usepackage{subfigure}
\usepackage{float}
\usepackage{cancel}
\usepackage{ulem}
\usepackage{bm}
\usepackage{amsfonts,amssymb,amsthm}

\newtheorem{theorem}{Theorem}[section]
\newtheorem{lemma}[theorem]{Lemma}

\newtheorem{remark}[theorem]{Remark}

\newtheorem{assumption}[theorem]{Assumption}
\numberwithin{equation}{section}
\counterwithin{figure}{section}
\counterwithin{table}{section}

\usepackage{color}

\newcommand{\Norm}[1]{{\left\|{#1} \right\|}}
\newcommand{\SemiNorm}[1]{{\left|{#1} \right|}}
\newcommand{\jump}[1]{\left[\!\left[#1\right]\!\right]}

\newcommand{\calT}{\mathcal{T}}
\newcommand{\calC}{\mathcal{C}}
\newcommand{\calTh}{\calT_h}
\newcommand{\calE}{\mathcal{E}}
\newcommand{\calEh}{\calE_h}
\newcommand{\calEhI}{\calE_h^I}
\newcommand{\calEhB}{\calE_h^B}

\newcommand{\calN}{\mathcal{N}}
\newcommand{\calNh}{\calN_h}
\newcommand{\calNha}{\calN_h^a}
\newcommand{\calNhb}{\calN_h^b}
\newcommand{\EcalE}{\mathcal{E}^{\E}}
\newcommand{\EcalEc}{\EcalE_{c}}
\newcommand{\EcalEs}{\EcalE_{s}}
\newcommand{\calM}{\mathcal{M}}

\newcommand{\Vh}{{V}_h}
\newcommand{\VhE}{\Vh(\E)}
\newcommand{\Vho}{V_h^0}
\newcommand{\Rbb}{\mathbb{R}}
\newcommand{\Pbb}{\mathbb{P}}
\newcommand{\Nbb}{\mathbb{N}}
\newcommand{\E}{K}
\newcommand{\e}{e}
\newcommand{\hE}{h_\E}
\newcommand{\he}{h_\e}
\newcommand{\Ie}{I_\e}
\newcommand{\xbm}{\bm{x}}
\newcommand{\bbm}{\bm{b}}
\newcommand{\xibm}{\bm{\xi}}
\newcommand{\qbm}{\bm{q}}

\newcommand{\qbmk}{\qbm_{k-1}}
\newcommand{\Dbm}{\bm{D}}
\newcommand{\drm}{\mathrm{d}}
\newcommand{\dx}{\drm\xbm}
\newcommand{\ds}{\drm s}
\newcommand{\nbf}{\mathbf n}
\newcommand{\nbfE}{\nbf_\E}
\newcommand{\nbfEp}{\nbf_{\E^+}}
\newcommand{\nbfEm}{\nbf_{\E^-}}
\newcommand{\nbfe}{\nbf_\e}
\newcommand{\eh}{e_h}
\newcommand{\vh}{v_h}
\newcommand{\vpi}{v_\pi}
\newcommand{\vI}{v_I}
\newcommand{\uh}{u_h}
\newcommand{\upi}{u_\pi}
\newcommand{\uI}{u_I}
\newcommand{\ah}{a_h}\newcommand{\bh}{b_h}\newcommand{\ch}{c_h}\newcommand{\Bh}{B_h}
\newcommand{\ahE}{\ah^{\E}}\newcommand{\bhE}{\bh^{\E}}\newcommand{\chE}{\ch^{\E}}\newcommand{\BhE}{\Bh^{\E}}
\newcommand{\aE}{a^{\E}}\newcommand{\bE}{b^{\E}}\newcommand{\cE}{c^{\E}}
\newcommand{\SE}{S^{\E}}
\newcommand{\Pbbtilde}{\widetilde \Pbb}
\newcommand{\boldalpha}{\boldsymbol \alpha}

\newcommand{\Pitilde}{\widetilde \Pi}
\newcommand{\Pitildenabla}{\Pitilde_{k}^{\nabla,\E}}
\newcommand{\Pitildeze}{\Pitilde_{k-1}^{0,\e}}
\newcommand{\Pitildeo}{\Pitilde_{k-1}^{0,\E}}
\newcommand{\Pitildeoh}{\Pitilde_{k-1}^{0,h}}
\newcommand{\Pik}{\Pi_{k-1}^{0,\E}}
\newcommand{\Pikh}{\Pi_{k-1}^{0,h}}
\newcommand{\qtilde}{\widetilde q}
\newcommand{\EEk}{\calE_{\E}^{k-1}}
\newcommand{\Etildek}{\widetilde{{\calE}}_{\E}^{k-1}}
\newcommand{\qn}{q_n}
\newcommand{\vIE}{v_I^\E}
\newcommand{\PipartialE}{\Pi_0^{0,\partial\E}}
\title{Nonconforming virtual element method for general second-order
elliptic problems on curved domain}
\author{\normalsize{
Yi Liu\thanks{Dipartimento di Matematica e Applicazioni, Università degli Studi di Milano-Bicocca, 20125 Milano, Italy,
\tt{yi.liu1@unimib.it, alessandro.russo@unimib.it}},
Alessandro Russo\footnotemark[1]~\thanks{IMATI-CNR, 27100, Pavia, Italy}}}

\date{}

\begin{document}

\maketitle
\begin{abstract}
\noindent This paper introduces a nonconforming virtual element method for  general second-order elliptic problems with variable coefficients on domains with curved boundaries and curved internal interfaces. 
We prove arbitrary order optimal convergence in the energy and~$L^2$ norms, 
confirmed by numerical experiments on a set of polygonal meshes. 
The accuracy of the numerical approximation provided by the method is shown to be comparable with the theoretical analysis.
\end{abstract}

\section{Introduction}
The present work proposes a nonconforming virtual element method for meshes with elements having curved edges for general second-order elliptic problems. 
The method allows to handle domains with curved boundaries or interfaces. To illustrate the main idea, we consider the Drichlet problem for second-order elliptic equations which seek an unknown function~$u$ satisfying 

\begin{equation}\label{eqn:primal problem}
 \left\{
 \begin{aligned}
 & -\nabla\cdot (a\nabla u)+\nabla\cdot(\bbm u)+c u = f \quad \text{in}~\Omega,\\
 &u=g\quad\text{on}~\partial\Omega,
 \end{aligned}\right.
\end{equation}
where $\Omega$ is a domain with curved boundaries or interfaces in~$\Rbb^d$ ($d=2,~3$), 
~$a = (a_{i,j}(\xbm))\in [L^{\infty}(\Omega)]^{d^2}$ is a symmetric matrix-valued function,~$\bbm= (b_i(\xbm))_{d\times1}$ is a vector-valued function, and $c= c(\xbm)$ is a scalar function on~$\Omega$. Assume that the matrix $a$ is strongly elliptic: there exists $a_0>0$ such that 
\[a_0\xibm^T\xibm \leq \xibm^T a\xibm, \quad \forall \xibm\in \Rbb^d.\] 

Dealing with curved boundaries for approximation degrees greater than one has some pitfalls.
Indeed, the discrepancy between the curved physical domain and the polygonal approximation domain leads to a loss of accuracy for
discretization with polynomial order~$k\geq 2$~\cite{Strang-Berger:1973,Thomee:1973}. 
To address this issue, 
one straightforward approach is to reduce the geometric error down to the same level of the approximation error. 
Popular methods following this track include the isoparametric finite element method~\cite{Ergatoudis-Irons-Zienkiewicz:1968, Lenoir:1986}, 
which requires a polynomial approximation of the curved boundary and careful selection of the isoparametric nodes; another notable approach applicable to CAD domains is Isogeometric Analysis~\cite{Cottrell-Hughes-Bazilevs:2009}.

Compared to traditional finite element methods on simplex, in this work, 
we research a numerical method on general polygonal and polyhedral meshes. 
Indeed, it is evident that with polytopal elements one can easily handle hanging nodes, movable meshes and adaptivity in an automatic manner. 
In the last decade, the curved edge discretizations have been investigated in the context of polytopal methods. 
Including the extended hybridizable discontinuous Galerkin method~\cite{Gurkan-SalaLardies-Kronbichler-FernandezMandez:2016};
the unfitted hybrid high-order method~\cite{Burman-Ern:2019,Burman-Cicuttin-Delay-Ern:2021};
the hybrid high-order method for the Poisson~\cite{Botti-DiPietro:2018, Yemm:2023}
and (singularly perturbed) fourth order problems~\cite{Dong-Ern:2022};
the Trefftz-based finite element method~\cite{Anand-Ovall-Reynolds-Weisser:2020};
the weak Galerkin method~\cite{liu-chen-wang:2022}. 

We then present a brief overview of the virtual finite element method. The original virtual element method~\cite{BeiradaVeiga-Brezzi-Cangiani-Manzini-Marini-Russo:2013,BeiradaVeiga-Brezzi-Marini-Russo:2014} is a~$\calC^0$-conforming method for solving the two-dimensional Poisson equation. 
A nonconforming counterpart for solving the same problem was presented in~\cite{AyusodeDios-Lipnikov-Manzini:2016}. 
The extension of these methods to general elliptic problems with variable coefficients is nontrivial. 
While an explicit representation of basis functions is not required, a crucial component of the former schemes is the Galerkin projection operator.
In contrast, the conforming and nonconforming virtual element approaches for general elliptic problems~\cite{Beirao-Brezzi-Marini-Russo:2016} and~\cite{cangiani-Manzini-Sutton:2017} use the~$L^2$-projection operators as designed in~\cite{A-A-B-M-R:2013}.

The virtual element methods have been developed to handle curved edges. In~\cite{BeiraodaVeiga-Russo-Vacca:2019}, the approach is defined (in 2D) directly on curved mesh elements, with the restrictions of their basis functions on curved edges are mapped polynomials. 
A modification proposed later~\cite{Beirao-Brezzi-Marini-Russo:2020}, applicable only to the 2D case, uses the restriction of physical polynomials to define shape functions on curved boundary edges. 
In~\cite{Bertoluzza-Pennacchio-Prada:2019}, the boundary correction technique~\cite{Bramble-Dupont-Thomee:1972} was applied to the virtual element method. 
The idea is to use normal-directional Taylor expansion, in most cases just a linear approximation, to correct function values on the boundary.
 The ideas of~\cite{BeiraodaVeiga-Russo-Vacca:2019} have been generalized to the approximation of solutions to the wave equations. 
 Mixed virtual element methods on curved domains were proposed in~\cite{Dassi-Fumagalli-Losapio-Scialo-Scotti-Vacca:2021}~(2D) and~\cite{Dassi-Fumagalli-Scotti-Vacca:2022}~(3D), respectively.
Additionally, we mention an earlier virtual element method on surfaces~\cite{Frittelli-Sgura:2018, Frittelli-Madzvamuse-Sgura:2021}.

Here, we extend the nonconforming virtual element method to curved edge elements and general second-order elliptic problems,
following the approach proposed in~\cite{AyusodeDios-Lipnikov-Manzini:2016} and \cite{Beirao-Brezzi-Marini-Russo:2016}. 
When the curved boundaries happen to be straight, the proposed virtual space reduces to that in~\cite{AyusodeDios-Lipnikov-Manzini:2016}. 
Although the nonconforming virtual element method for general linear elliptic problems with variable coefficients was designed in~\cite{cangiani-Manzini-Sutton:2017}, 
the bilinear forms presented in this paper differ from those in previous work, 
specifically addressing curved edges.

The paper is organized as follows.
In Section~\ref{section:modelproblem}, we introduce the model problem and outline some notations and hypotheses on curved edges. 
Section~\ref{section:Meshes and broken spaces} presents the mesh assumptions, introduces the discrete spaces along with the associated degrees of freedom, defines the discrete bilinear forms, and formulates the discrete problem.
In Section~\ref{section:Polynomial and virtual element approximation estimates}, we rigorously analyze the theoretical properties of the proposed projection and interpolation operators.
In Section~\ref{section:stability},
 we prove the existence and uniqueness for nonconforming virtual element approximations. 
 In Section~\ref{section:Convergence analysis}, we recover the optimal order of convergence in the energy and $L^2$ norms for the proposed method.
 Details on the 3D version of the method are discussed
in Section~\ref{section:3D}.
Finally, numerical results are presented in Section~\ref{section:Numerical experiments}.

\section{Notation and preliminaries}\label{section:modelproblem}
In this section, we first introduce the standard notation for Sobolev spaces and norms~\cite{Adams:2003}. 
Given an open Lipschitz domain~$D\subset \Rbb^d$,~$d\in\Nbb$, 
the standard norms in the space~$W^{m,p}(\Omega)$ and~$L^p(D)$ are denoted by~$\Norm{\cdot}_{W^{m,p}(D)}$ and~$\Norm{\cdot}_{L^p(D)}$ respectively. 
Norm and seminorm in~$H^m(D)$ are denoted respectively by~$\Norm{\cdot}_{m,D}$ and~$\SemiNorm{\cdot}_{m,D}$, 
while~$(\cdot,\cdot)_{D}$ and~$\Norm{\cdot}_{D}$ denote the~$L^2$-inner product and~$L^2$-norm.   
Sobolev spaces of negative order can be defined by duality.
The notation $\langle\cdot,\cdot\rangle$ stands for the duality pairing~$H^{-\frac12} - H^{\frac12}$ on a given domain.

Next, according to integration by parts, the weak formulation of problem~\eqref{eqn:the weak form} is given as follows: 

seeks~$u\in H^1(\Omega)$ such that~$u=g$ on~$\partial\Omega$ and
 \begin{equation}\label{eqn:weak problem}
 \begin{aligned}
 a(u,v)+ b(u,v)+ c(u,v)  = F(v)\quad \forall v\in H_0^1(\Omega). 
 \end{aligned}
\end{equation}
where the bilinear and linear forms are defined as
\begin{equation}\label{eqn:bilinear terms}
 \begin{aligned}
 &a(u,v):= \int_\Omega{a \nabla u \cdot\nabla v}~\drm \xbm, \quad b(u,v) := - \int_\Omega{u\bbm\cdot\nabla v}~\drm \xbm,\\
 &c(u,v):= \int_\Omega{c u v}~\drm \xbm,\quad F(v) := \int_\Omega{fv }~\drm \xbm.
 \end{aligned}
\end{equation}
and 
\begin{equation}\label{eqn:the weak form}
 \begin{aligned}
 B(u,v) := a(u,v)+ b(u,v)+ c(u,v).
  \end{aligned}
\end{equation}

Finally, we present the assumption of the boundary~$\partial \Omega$~\cite{Beirao-Liu-Mascotto-Russo:2023}: it is the union
a finite number of smooth curved edges/faces
$\{\Gamma_i\}_{i=1,\cdots,N}$, i.e.,
\[
\bigcup_{i=1}^N\Gamma_i = \partial\Omega.
\]
Each $\Gamma_i$
is of class $\mathcal{C}^{\eta}$, for an integer $\eta \geq 1$,
which will fixed in Assumption~\ref{assumption:eta} below:
there exists a given regular and invertible $\mathcal{C}^{\eta}$-parametrization
$\gamma_i:I_i(F_i)\rightarrow\Gamma_i$ for $i = 1, \dots , N$, 
where $I_i$~(2D) and ~$F_i$~(3D) are a closed interval and a straight polygon,respectively.
The smoothness parameter~$\eta$ depends on the order of the numerical scheme
and will be specified later.

Since all the~$\Gamma_i$ can be treated analogously in the forthcoming analysis,
 we drop the index~$i$
and assume that~$\partial \Omega$ contains only one curved face~$\Gamma$.
To further simplify the presentation,
we focus on the two dimensional case
and postpone the discussion of the three dimensional case to Section~\ref{section:3D} below.
For the sake of notation simple, we further assume that $\gamma:[0,1] \to \Gamma$.
\begin{remark} \label{remark:internal-interfaces}
The forthcoming analysis can be extended to
the case of fixed internal interfaces (and jumping coefficients) with minor modifications.
To simplify the presentation, we stick to the case of~$\Gamma$ being a curved boundary face;
however, we shall present numerical experiments for jumping coefficients
across internal curved interfaces.
\end{remark}

 \section{The nonconforming virtual element method }\label{section:Meshes and broken spaces}
 
\subsection{Mesh assumptions} \label{subsection:meshes}
 Let~$\calTh$ be a decomposition of~$\Omega$. 
 We denote the diameter of each element~$\E$ by~$\hE$,
and the mesh size function of~$\calTh$ by~$h := \max_{\E \in \calT_h}\hE$. 
And let~$\calEh$ be the set of edges e of $\calEh$. 
We denote the size of any edge~$\e$ by~$\he$.
Let $\calEhI$ and $\calEhB$ be the sets of all interior and boundary edges in~$\calTh$.

For each element $\E\in\calTh$,
we denote the sets of its edges by~$\EcalE$,
which we split into straight~$\EcalEs$
and curved edges~$\EcalEc$, respectively.
With each element~$\E$, we associate the outward unit normal vector~$\nbfE$;
with each edge~$\e$, we associate a unit normal vector~$\nbfe$ out of the available two.

 Henceforth, we demand the following regularity assumptions on the sequence~$\calTh$:
there exists a positive constant~$\rho$ such that
\begin{itemize}
\item[\textbf{(G1)}] each element~$\E$ is star-shaped with respect to a ball of radius larger than or equal to~$\rho \hE$;
\item[\textbf{(G2)}] for each element~$\E$ and any of its edges (possibly curved)~$\e$,
$\he$ is larger than or equal to~$\rho \hE$.
\end{itemize}

We introduce a parametrization of the edges: For a curved edge, we have~$\gamma_e : I_e\rightarrow e$ is a restriction of the global parametrization~$\gamma$, while for a non curved edge~$e$,~$\gamma_e$ is an affine map.   

We shall write $x\lesssim y$ and $x\gtrsim y$ instead of $x\leq Cy$ and $x\geq Cy$, respectively,
for a positive constant $C$ independent of~$\calTh$.
Moreover, $x\approx y$ stands for 
$x \lesssim y$ and~$y \lesssim x$ at once.
The involved constants will be written explicitly only when necessary.

The validity of (\textbf{G1})--(\textbf{G2})
guarantees that the constants in the forthcoming trace and inverse inequalities are uniformly bounded.
\subsection{Broken spaces} \label{subsection:broken-spaces}
 For any integer $n\geq -1$ and any element~$\E\in\calTh$, 
 we define~$\Pbb_n(\E)$ be the set of polynomials on~$\E$ of degree less or equal to~$n$.
 In the case~$n=-1$ we set~ $\Pbb_{-1}(\E) = \{0\}$.
Indentify with~$\bm{x}_{\E}$ the centroid of the element~$\E$,
we introduce the space of normalized monomials 
 as
\[
\calM_n(\E) =
\left\{
\left(\dfrac{\bm{x}-\bm{x}_{\E}}{\hE}\right)^{\boldalpha}
\; \forall  \boldalpha \in \Nbb^2, \; \vert \boldalpha \vert \le n,\;\;
\forall \bm x \in \E \right\}.
\]
The space~$\calM_n(\E)$ forms a basis for~$\Pbb_n(\E)$.
For the edges of the grid, we introduce approximation spaces that consider the curved gemometry. 
For a reference segment~$I_e$, we introduce the monomial set
\[
\calM_n(\Ie) =
\left\{
\left(\dfrac{{x}-{x}_{\Ie}}{h_{\Ie}}\right)^{\alpha}
\; \forall  \alpha \in \Nbb, \;  \alpha \le n,\;\;
\forall x \in \Ie \right\}.
\]
with~$x_{I_e}$ the midpoint of~$I_e$ its size.
Next, we define the mapped polynomial spaces on the edges in~$\calE_h$, given by
\[
\begin{split}
\Pbbtilde_n(\e)   
= \{\qtilde 
= q\ \circ\ \gamma_\e^{-1}:\ q\in {P}_n(\Ie)\},
\qquad
\widetilde{\calM}_n(\e)  = \{\widetilde{m} = m\ \circ\ \gamma_\e^{-1}:\ m \in \calM_n(\Ie)\}.
\end{split}
\]
where~$\gamma_e$ represents the local map of the edge~$I_e$ to~$e$ as discussed before.

If~$\e$ is straight, then $\Pbbtilde_n(\e)$ and~$\widetilde{\calM}_n(\e)$
boil down to a standard polynomial space
and scaled monomial set, respectively.
For any~$s>0$, we introduce the broken Sobolev space over a mesh~$\calTh$ as
\[
H^s(\calT_h) := \{ v\in L^2(\Omega):\, v_{|\E} \in H^s(\E)\quad \forall \E\in\calTh\}
\]
and equip it with the broken norm and seminorm
\[
\Norm{v}_{s,h}^2 := 
\sum_{\E\in\calT_h}\Norm{v}_{s,\E}^2 ,
\qquad
\SemiNorm{v}_{s,h}^2 :=
\sum_{\E\in\calT_h}\SemiNorm{v}_{s,\E}^2 .
\]
We define the jump across the edge~$\e$ of any~$v$ in~$H^1(\calTh)$ as
\[
\jump{v}:=
\begin{cases}
v_{|\E^+}\nbfEp +v_{|\E^-} \nbfEm   & \text{if~$\e\in\calE_h^I$,  $\e \subset \partial\E^+ \cap \partial \E^-$ for given~$\E^+$, $\E^- \in \calTh$} \\
v\nbfe                               & \text{if~$\e\in\calEhB$, $\e \subset \E$ for a given~$\E \in \calTh$.}
\end{cases}
\]
The nonconforming Sobolev space of order~$n\in\Nbb$ over~$\calTh$ is given as follows:
\[
H^{1,nc}(\calTh,n) 
:= \left\{
v\in H^1(\calTh) \middle|  
\int_{\e}\jump{v}\cdot\nbfe \widetilde{q}_{n-1}~\drm s=0,\;
\forall \widetilde{q}_{n-1}\in \widetilde{\Pbb}_{n-1}(\e),\ \forall e\in\calEhI
\right\}.
\]
\subsection{The nonconforming virtual element space} \label{subsection:Nonconforming spaces}

\paragraph*{Projections onto polynomial spaces.}
In this section, we introduce the operators that are essential for the implementation of the method discussed subsequently.
On each element~$\E$, we have 

\begin{itemize}
\item the (possibly curved) edge $L^2$ projection $\Pitilde_{n}^{0,\e} : L^2(\e)\rightarrow \Pbbtilde_n(\e)$
given by
\begin{equation} \label{L2-projection-edge}
\int_\e\qtilde_n  (v -\Pitilde_n^{0,\e} v )\ \mathrm{d}s =0
\qquad \forall \qtilde_n\in  \Pbbtilde_n(\e);
\end{equation}
\item the Ritz-Galerkin operator
$\Pitilde_n^{\nabla,\E} : H^1(\E)\rightarrow \Pbb_n(\E)$
satisfying, for all~$\qn$ in~$\Pbb_n(\E)$,
\begin{equation} \label{eqn:RG-projection-bulk}
\begin{split}
& \int_\E \nabla \qn \cdot \nabla\Pitilde_n^{\nabla,\E} v \ \drm \xbm
= -\int_\E\Delta \qn  v \ \drm \xbm
+\sum_{\e\in\calE^K} \int_\e\Pitilde_{n-1}^{0,\e}
(\nbfE \cdot \nabla \qn) v \ \drm s\\ 
& = -\int_\E\Delta \qn  v \ \drm \xbm
+\sum_{\e\in\calE_{s}^K}\int_\e 
(\nbfE \cdot \nabla \qn)  v \ \drm s
+\sum_{\e\in\calE_{c}^K}\int_\e\Pitilde_{n-1}^{0,\e}
(\nbfE \cdot \nabla \qn) v \ \drm s,
\end{split}
\end{equation}
together with
\begin{equation} \label{eqn:RG-projection-zero-average}
\left\{
\begin{aligned}
\int_{\partial \E} (v -\Pitilde_n^{\nabla,\E} v )\ \mathrm{d}s =0 \quad n=1,\\
\int_{\E} (v -\Pitilde_n^{\nabla,\E} v )\ \mathrm{d}\xbm =0\quad n\geq 2;
\end{aligned}
\right.
\end{equation}
\item the $L^2$ operator $\Pitilde_{n}^{0,K}:\nabla (H^1(\E))\rightarrow [\Pbb_n(\E)]^d$ satisfying, for all $\qbm_n \in [\Pbb_n(\E)]^d$,
\begin{equation} \label{eqn:L2-tilde-projection-bulk}
\begin{split}
&\int_\E \qbm_n \cdot (\Pitilde_{n}^{0,K} \nabla v) \ \mathrm{d}\xbm 
= -\int_\E \nabla \cdot\qbm_n    v \ \mathrm{d}\xbm +\sum_{\e\in\calE^K} \int_\e\Pitilde_{n}^{0,\e}
(\nbfE \cdot \qbm_n) v \ \mathrm{d}s\\
& = -\int_\E \nabla \cdot\qbm_n    v \ \mathrm{d}\xbm
+\sum_{\e\in\calE_{s}^K}\int_\e 
(\nbfE \cdot \qbm_n) v  \ \mathrm{d}s
+\sum_{\e\in\calE_{c}^K}\int_\e\Pitilde_{n}^{0,\e}
(\nbfE \cdot \qbm_n) v  \ \mathrm{d}s.
\end{split}
\end{equation}
\item the $L^2$ projection
$\Pi_n^{0,\E} : L^2(\E)\rightarrow \Pbb_n(\E)$ given by
\begin{equation} \label{eqn:L2-projection-bulk}
\int_\E \qn  (v -\Pi_n^{0,\E} v )\ \mathrm{d}\xbm =0
\qquad \forall \qn \in \Pbb_n(\E). 
\end{equation}
For simplicity, we will use the same notation for the~$L^2$-projection of vector-valued functions onto the polynomial space~$[\Pbb_n(\E)]^d$. 
\end{itemize}
\medskip
\paragraph*{ Approximation space and degrees of freedom.}
Let $k\geq 1$ be the polynomial order of the method. By following the approach derived in~\cite{Beirao-Liu-Mascotto-Russo:2023} for elliptic problems, we select the following enhanced local virtual element space defined as:

\begin{equation}\label{eqn:local virtual space}
 \begin{aligned}
 \VhE:=\{\vh \in H^1(\E):~&\Delta \vh \in \Pbb_{k}(\E),~
\nbfE \cdot \nabla \vh \in \Pbbtilde_{k-1}(\e)\ 
~\forall \e \in \EcalE, \\
&(\Pitilde_{k}^{\nabla,\E} v-v,q_k)_{\E}=0~\forall q_k\in (\Pbb_k(\E)/\Pbb_{k-2}(\E))\}. 
 \end{aligned}
\end{equation}
We remark that, if~$\E$ is an element with only straight edges, then~\eqref{eqn:local virtual space} reduces to the enhanced virtual element spaces as defined in~\cite{cangiani-Manzini-Sutton:2017}.
In general, the space~$\VhE$ does not give a closed form for computing its shape functions. 
We here summarize the main properties of the space~$\VhE$ (we refer to \cite{A-A-B-M-R:2013,Beirao-Liu-Mascotto-Russo:2023} for a deeper
analysis)

\begin{itemize}

\item\textbf{Polynomial inclusion:}~$\Pbb_0(\E)\subset \VhE$ but in general~$\Pbb_k(\E)\not\subset\VhE$ on the element with some curved edges.
\item \textbf{Degrees of freedom:}~the following linear operators constitute a set of DoFs for~$\VhE$: for any~$\vh\in\VhE$ we consider 
\begin{itemize}
\item[-]on each edge~$\e$ of~$\E$, the moments
\begin{equation}\label{eqn:edge-dofs}
\bm{D}_{\e}^{i}(\vh) 
= |\e|^{-1} \int_\e \vh \widetilde{m}_i\ \mathrm{d}s\qquad
\forall \widetilde{m}_i \in \widetilde{\calM}_{k-1}(\e);
\end{equation}
\item[-] the bulk moments
\begin{equation}\label{eqn:bulk-dofs}
\bm{D}_{\E}^{j}(\vh)
= |\E|^{-1} \int_\E \vh m_j\ \mathrm{d}\xbm
\qquad  \forall m_j \in \calM_{k-2}(\E).
\end{equation}
\end{itemize}

\item \textbf{Polynomial projections:}
the operators~$\Pitilde_k^{\nabla,\E}$ and~$\Pi_k^{0,\E}$  can be explicitly computed from the DoFs ($\Dbm$)(see~\cite{Beirao-Liu-Mascotto-Russo:2023} and \cite{A-A-B-M-R:2013} );
 and~the function~$\Pitilde_{k-1}^{0,\E}\nabla\vh $~also can be explicitly computed from the DoFs ($\Dbm$) of~$\vh$(see below).

While the computation of ~the function~$\Pitilde_{k-1}^{0,\E}\nabla\vh $~is well-documented in the literature, 
the detailed process for computing~$\Pitilde_{k-1}^{0,\E}\nabla v$ requires further discussion.
To compute~$\Pitilde_{k-1}^{0,\E}\nabla v$, for all~$\E\in\calTh$ we need to calculate 
\[-\int_\E v \nabla\cdot \qbm_{k-1}\ \drm \xbm +\sum_{\e\in\calE_{s}^K}\int_\e 
(\nbfE \cdot  \qbm_{k-1})  v \ \drm s+\sum_{\e\in\calE_{c}^K}\int_\e\Pitilde_{k-1}^{0,\e}(\qbm_{k-1}\cdot\nbf)v \ \drm s \quad\forall \qbm_{k-1}\in[\Pbb_{k-1}]^2(\E).\] 
The first term above clearly depends only on the moments of~$v$ appearing in \eqref{eqn:bulk-dofs};
The second and third terms are computable from \eqref{eqn:edge-dofs}.

\end{itemize}

The global nonconforming virtual element space is obtained by a standard coupling of the interface degrees of freedom~\eqref{eqn:edge-dofs}:
\begin{equation}\label{eqn:global-space}
\Vh(\calTh)
:=
\left\{
\vh \in H^{1,nc}(\calTh,k) \middle|
{\vh}_{| \E} \in \VhE ~ \forall\E\in\calTh
\right\}.
\end{equation}

We further define nonconforming virtual element spaces with weakly imposed boundary conditions:
for all~$g$ in~$L^2(\partial \Omega)$,
\begin{equation}\label{eqn:Vkg}
\Vh^g(\calTh)
:= \left\{
\vh \in \Vh(\calTh) \middle| 
\int_\e (\vh-g) \qtilde_{k-1}\  \drm s= 0 \quad
\forall \qtilde_{k-1} \in \widetilde{\Pbb}_{k-1}(\e),\ \forall\e\in\calEhB
\right\}.
\end{equation}

\subsection{The nonconforming virtual element method}
\label{subsection:polynomial-projections-bf}

In this section, we present the discrete weak formulation of problem~\eqref{eqn:weak problem} using the operators and functional spaces introduced earlier.
\paragraph*{ Discrete bilinear forms.}
Given an element~$\E$,
we define the discrete local bilinear forms, 
for any~$\uh,\ \vh \in \VhE+\Pbb_k(\E)$ as follows:
\begin{equation}\label{eqn:local bilinear forms}
\begin{aligned}
\ahE(\uh,\vh) &= \int_\E (a\Pitildeo\nabla\uh)\cdot( \Pitildeo  \nabla\vh)~\drm  \xbm 
+\SE((I-\Pitilde_k^{\nabla,\E} )\uh,(I-\Pitilde_k^{\nabla,\E} )\vh);\\
\bhE(\uh,\vh) &= -\int_\E\Pi_{k-1}^{0,\E}  \uh (\bbm\cdot\Pitildeo  \nabla\vh)~\drm  \xbm  ;\\
\chE(\uh,\vh) &= \int_\E c(\Pi_{k-1}^{0,\E}  \uh)( \Pi_{k-1}^{0,\E}  \vh)~\drm  \xbm  ;\\
\BhE(\uh,\vh)& = \ahE(\uh,\vh)+\bhE(\uh,\vh)+\chE(\uh,\vh) .
\end{aligned}
\end{equation}

In this paper, we use~$\SE(\cdot,\cdot) : \VhE\times \VhE\rightarrow \Rbb$ defined as
\begin{equation}\label{eqn:dofi-dofi}
\SE(\uh,\vh) = \sum_{l=1}^{N_{\mathrm{dof}}(\E)}
\bm{D}^l(\uh)\bm{D}^l(\vh),
\end{equation}
being~$N_{\mathrm{dof}}(\E)$ the number of DoFs on~$\E$.
More precisely, we have the continuity and coercivity of the stabilization~$\SE$ (see Proposition 2 and 3 in~\cite{Beirao-Liu-Mascotto-Russo:2023}): There exists positive constants~$\alpha$ and~$\alpha^*$ only depend on~$k$ and the shape regularity constant~$\rho$, such that
\begin{equation}\label{eqn:stability bounded property}
  \SE(\uh,\vh)
\leq \alpha^*\SemiNorm{\uh}_{1,\E}\SemiNorm{\vh}_{1,\E}\qquad\forall \vh \in (\VhE+\Pbb_k(\E)) \cap \ker(\Pitilde_k^{\nabla,\E}), 
\end{equation}
and 
\begin{equation} \label{eqn:stability-bounds}
\alpha\SemiNorm{\vh}_{1,\E}^2
\leq \SE(\vh,\vh)
\qquad\forall \vh \in \VhE .
\end{equation}

Of course, other stabilizations satisfying the continuity and stability properties can be defined;
we stick to the choice in~\eqref{eqn:dofi-dofi} since it is the most popular in the virtual element community.
\paragraph*{The global bilinear terms.} We set for all~$\vh\in\Vh+\Pi_{\E\in\calTh}\Pbb_k(\E)$: 
\begin{equation}\label{eqn:bilinear forms}
\begin{aligned}
\ah(\uh,\vh) &= \sum_{\E\in\calTh}\ahE(\uh,\vh) ;&\bh(\uh,\vh) = \sum_{\E\in\calTh}\bhE(\uh,\vh);\\
\ch(\uh,\vh) &=\sum_{\E\in\calTh}\chE(\uh,\vh);&\Bh(\uh,\vh) = \sum_{\E\in\calTh}\BhE(\uh,\vh).
\end{aligned}
\end{equation}
\paragraph*{The discrete right-hand side.}
Here, we construct a computable discretization of the right-hand side~$F_h(\vh)$ as discussed in~\eqref{eqn:weak problem}~\cite{Beirao-Brezzi-Marini-Russo:2016}.
\begin{equation} \label{eqn:fk1}
F_h(\vh)
= \sum_{\E\in\calTh}\int_\E  f\ \Pi_{k-1}^{0,\E}\vh \drm \xbm 
 = \sum_{\E\in\calTh}\int_\E  \Pi_{k-1}^{0,\E}f\ \vh \drm \xbm 
\qquad \qquad \forall \vh \in \Vh;
\end{equation}

We propose the following nonconforming virtual element method:
\begin{equation}\label{eqn:VEM}
\begin{cases}
& \text{find $\uh\in \Vh^g(\calTh)$ such that} \\
& \Bh (\uh,\vh)= F_h(\vh)
        \qquad\qquad \forall \vh\in \Vh^0(\calTh).
\end{cases}
\end{equation}
The well posedness of method~\eqref{eqn:VEM}
requires further technical tools,
which are derived in section~\ref{section:stability}.

\section{Polynomial and virtual element approximation estimates} \label{section:Polynomial and virtual element approximation estimates}

To analyze the stability and the approximation properties of the local discrete bilinear form, we need the following result,
which states the stability of the operators~$\Pitildenabla$ and~$\Pitildeo$ in~$H^1$ and~$L^2$ norms, as defined in~\eqref{eqn:RG-projection-bulk},~\eqref{eqn:RG-projection-zero-average} and~\eqref{eqn:L2-tilde-projection-bulk}, respectively.

Consider the operator~$\PipartialE : H^1(\E) \to \Rbb$ given by
\[
\PipartialE v := \frac{1}{|\partial \E|}\int_{\partial \E} v\ \mathrm{d}s.
\]
The following Poincar\'e inequality is valid:
\begin{equation}\label{eqn:poincare}
\|v-\PipartialE v\|_{0,\E}\lesssim \hE\SemiNorm{v}_{1,\E}.
\end{equation}
\begin{lemma}\label{lemma:PTcont}
Given an element~$\E$ and $v_h$ in~$\VhE$, we have
\begin{equation}\label{eqn:tpi0bd}
\begin{aligned}
\SemiNorm{\Pitildenabla\vh}_{1,\E}
\lesssim \SemiNorm{ v_h}_{1,\E};\\
  \Norm{\Pitildeo \nabla v_h}_{0,\E}
\lesssim \Norm{\nabla v_h}_{0,\E}.  
\end{aligned}
\end{equation}
\end{lemma}
\begin{proof}
The stability of~$\Pitildenabla$ follows immediately from Lemma 6 in \cite{Beirao-Liu-Mascotto-Russo:2023}
and is therefore not detailed here. 
We will, however, focus on the second operator.

Introduce $\bar v_h = v_h-\PipartialE \vh$. Clearly,
we have $ \nabla \bar v_h  = \nabla\vh$.  
Therefore, by substituting $\qbm_k = \Pitildeo \nabla \bar v_h$ into \eqref{eqn:L2-tilde-projection-bulk}, we obtain
$$
\begin{aligned}
&\Norm{\Pitildeo \nabla \vh}_{0,\E}^2
=\Norm{\Pitildeo \nabla \bar v_h}_{0,\E}^2 \\
&=-\int_\E \bar v_h \nabla\cdot (\Pitildeo \nabla \bar v_h)\ \dx 
+\sum_{\e\in\calE_{s}^K}\int_\e 
(\nbfE \cdot  \Pitildeo \nabla \bar v_h)  \bar v_h \ \drm s\\
&\quad+\sum_{\e\in\calE_{c}^K}\int_\e\Pitildeze(\nbfE \cdot\Pitildeo \nabla \bar v_h)\bar v_h \ \drm s\\
&=\int_\E \nabla \bar v_h \cdot (\Pitildeo \nabla \bar v_h)\ \dx-\sum_{\e\in\calE_{c}^K}\int_\e \bar v_h(I-\Pitildeze)( \nbfE \cdot\Pitildeo \nabla \bar v_h) \ \drm s\\
&\leq \Norm{\nabla \bar v_h}_{0,\E}\Norm{\Pitildeo \nabla \bar v_h}_{0,\E}
+\sum_{\e\in\calE_{c}^K}\|\nbfE \cdot \Pitildeo \nabla \bar v_h \|_{0,\e}
 \|\bar v_h\|_{0,\e}.
\end{aligned}$$

Applying  standard trace inequality, 
the scaled Poincar\'e inequality and inverse inequality, 
we infer
\begin{equation}\label{trace-inverse-vh}
   \hE^{-\frac12}\Norm{\bar v_h}_{0,\e}
\lesssim \hE^{-1}\Norm{\bar v_h}_{0,\E}+\Norm{\nabla\bar v_h}_{0,\E}
\lesssim\Norm{\nabla\bar v_h}_{0,\E}, 
\end{equation}
and
\[
\hE^{\frac12}
\| \nbfE \cdot\Pitildeo \nabla \bar v_h \|_{0,\e}
\lesssim\Norm{ \Pitildeo \nabla \bar v_h  }_{0,\E}
+\hE\SemiNorm{\Pitildeo \nabla \bar v_h }_{1,\E}.
\]

By combining the three estimates above, we obtain
\[
\Norm{\Pitildeo \nabla \vh }_{0,\E}^2 
\lesssim\Norm{\nabla\bar v_h}_{0,\E}
\Norm{\Pitildeo \nabla  \bar v_h }_{0,\E}
\lesssim\Norm{\nabla\vh}_{0,\E}\Norm{\Pitildeo \nabla  \vh }_{0,\E}.
\]

The assertions follow from the above estimates.
\end{proof}

And then, we present the approximation properties
of the two following approximants for sufficiently regular functions~$v$:
\begin{itemize}
\item the~$L^2$ projection~$\vpi = \Pi_k^{0,\E}v$ of $v$ into the polynomial space~$\Pbb_k(\E)$;
\item the DoFs interpolant~$\vIE$ in~$\VhE$ of~$v$
defined as
\begin{equation}\label{eqn:vik}
\begin{cases}
\int_\E (\vIE-v) m\ \mathrm{d}K = 0
& \forall m \in \calM_{k-2}(\E) \\
\int_\e (\vIE-v) \widetilde{m}\ \mathrm{d}s  =0 
&  \forall \widetilde{m} \in \widetilde{\calM}_{k-1}(\e), \quad \forall e\in \EcalE.
\end{cases}
\end{equation}
\end{itemize}

\begin{assumption} \label{assumption:eta}
Henceforth, the regularity parameter~$\eta$ of the curved boundary
introduced satisfies $\eta \ge k$.
\end{assumption}
Polynomial error estimates for~$\vpi$ are well know; see e.g.\cite{Brenner-Scott:2008}:
\begin{equation}\label{eqn:Polynomial error estimates}
\|v-\vpi\|_{0,\E}+\hE|v-\vpi|_{1,\E}
\lesssim\hE^{k+1} |v|_{k+1,\E}
\qquad \forall v\in H^{k+1}(\E).
\end{equation}

Then, we recall the approximation properties of some projection operators proved in \cite{Beirao-Liu-Mascotto-Russo:2023}.


\begin{lemma}\label{lemma:proes}
Let~$n \in \Nbb$ and the regularity parameter~$\eta$ of the curved boundary satisfy~$\eta\ge n+1$.
Given an element~$\E$ and any of its edges~$\e$,
for all $0\leq m \leq s \leq n+ 1$, we have
\[
|v - \Pitilde_{n}^{0,\e}v|_{m,\e}
\lesssim \he^{s-m}\Norm{v}_{s,\e}
\qquad \forall v \in H^s(\e) .
\]
\end{lemma}
\begin{lemma}\label{lemma:v-pikv2}
For all $v$ in~$H^{s+1}(\Omega)$,
$1 \le s \le k$, we have
\begin{equation}\label{eqn:v-pikv}
|v-\Pitilde_k^{\nabla,h} v|_{1,h}
\lesssim h^{s} \Norm{v}_{s+1,\Omega}.
\end{equation}

\end{lemma}

\begin{lemma}\label{lemma:DoFs-interpolant-estimates}For all~$v$ in~$H^{s+1}(\Omega)$,
$1 \le s \le k$, it holds
\begin{equation}\label{eqn:Approximation}
\begin{aligned}
\Norm{v-\vI}_{0,\Omega}+h\SemiNorm{v-\vI}_{1,h}&\lesssim h^{s+1}\Norm{v}_{s+1,\Omega}.
\end{aligned}
\end{equation}
\end{lemma}
\begin{proof}
    Clearly, the result follows from Lemma 10 in~\cite{Beirao-Liu-Mascotto-Russo:2023}, which states that:
    \[
\SemiNorm{v-\vI}_{1,h}\lesssim h^{s}\Norm{v}_{s+1,\Omega}.
\]

By the definition of~$\vI$, for any element~$\E$ we have~$\PipartialE {v}=\PipartialE{\vI}$. 
Using the Poincar\'e inequality \eqref{eqn:poincare}, we then find
\[\Norm{v-\vI}_{0,\Omega}=(\sum_{\E\in\calT_h}\Norm{v-\vI}_{0,\E}^2)^{\frac12}\leq (\sum_{\E\in\calT_h}h^{2s}\SemiNorm{v-\vI}_{1,\E}^2)^{\frac12}\lesssim h\SemiNorm{v-\vI}_{1,h}.\]

The assertion follows from the above estimates.
\end{proof}

We also recall the properties of the Stein's extension operator~$E$ as presented in~\cite[Chapter~VI, Theorem~$5$]{Stein:1970}.
\begin{lemma} \label{lemma:Stein}
Given a Lipschitz domain~$\Omega$ in~$\Rbb^2$ and $s \in \Rbb$, $s \ge 0$,
there exists an extension operator
$E: H^s(\Omega)\rightarrow H^s(\Rbb^2)$
such that
\begin{equation}\label{eqn:Extension}
Ev |_\Omega = v
\quad\text{and}\quad
\|Ev\|_{s,\Rbb^d}\lesssim \Norm{v}_{s,\Omega}
\qquad\qquad \forall v\in H^s(\Omega).
\end{equation}
The hidden constant depends on~$s$
but not on the diameter of~$\Omega$.
\end{lemma}

Using similar argument techniques as those in Lemma 11 of~\cite{Beirao-Liu-Mascotto-Russo:2023}, we can derive the error estimate for the $L^2$-operator~$\Pitildeo$
defined in~\eqref{eqn:L2-tilde-projection-bulk}.

\begin{lemma}\label{lemma:v-pikt0v}
For all $v$ in~$H^{s+1}(\Omega)$,
$1 \le s \le k$, we have
\begin{equation}\label{eqn:v-pikt0v}
\Norm{\nabla v-\Pitildeoh \nabla v}_{0,\Omega}
\lesssim h^{s} \Norm{v}_{s+1,\Omega}.
\end{equation}

\end{lemma}
\begin{proof}
    By applying the triangle inequality and the approximation property~$ \Pi_n^{0,\E}$, we obtain
\[
\begin{aligned}
\Norm{\nabla v-\Pitildeo\nabla v}_{0,\E}
&\leq \Norm{\nabla v-\Pik \nabla v}_{0,\E}+\Norm{\Pik \nabla v - \Pitildeo \nabla v}_{0,\E}\\
&\lesssim \hE^s \SemiNorm{v}_{s+1,\E}
+ \Norm{\Pik \nabla v - \Pitildeo \nabla v}_{0,\E}.
\end{aligned}
\]

By definitions of~$\Pitildeo$ and~$\Pik$,
for all~$\qbmk$ in~$\Pbb_{k-1}(\E)$,
we can write
\[
\begin{aligned}
&\int_\E \qbmk \cdot  (\Pik \nabla v -\Pitildeo\nabla v)~ \dx\\
=& \int_\E \qbmk \cdot \nabla v~ \dx
+\int_\E\nabla\cdot \qbmk v \ \dx
-\sum_{\e\in\EcalE}\int_\e\Pitildeze
(\nbfE \cdot \qbmk) v \ \mathrm{d}s\\
 = &\sum_{\e\in\EcalE}\int_\e (I- \Pitildeze) (\nbfE \cdot  \qbmk) v~\drm s
= \sum_{\e\in\EcalEc}\int_\e 
(I- \Pitildeze) (\nbfE \cdot  \qbmk) v ~\drm s.
\end{aligned}
\]

 We set
\[\bm{\psi} 
:= (\Pik \nabla v - \Pitildeo \nabla v)
\in [\Pbb_{k-1}(\E)]^2.
\]

We rewrite the above identity for the choice
$\qbmk = \bm{\psi} $ and obtain
\begin{equation} \label{eqn:difference-boldpik-boldpitildek}
\begin{aligned}
& \Norm{\Pik \nabla v - \Pitildeo \nabla v}_{0,\E}^2 
= \sum_{\e\in\EcalE_c}\int_\e (I- \Pitildeze)(\bm{\psi}\cdot\nbfE) v~\mathrm{d}s\\
& =\sum_{\e\in\EcalE_c}\int_\e(I-{\Pitildeze})(\bm{\psi}\cdot\nbfE) (I- \Pitildeze)v~\mathrm{d}s \\
& \leq\sum_{\e\in\EcalE_c}\|(I-{\Pitildeze})(\bm{\psi}\cdot\nbfE) \|_{0,\e}\|(I- \Pitildeze) v\|_{0,\e}.
\end{aligned}
\end{equation}

According to (36) in \cite{Beirao-Liu-Mascotto-Russo:2023}, we have derived the following
\begin{equation}\label{eqn:psin2}
\|(I- {\Pitildeze})(\bm{\psi}\cdot\nbfE)\|_{0,\e} 
\lesssim\hE^{\frac12} \|{\bm{\psi}}\|_{0,\E}.
\end{equation}
On the other hand, Lemma~\ref{lemma:proes} implies
\begin{equation} \label{eqn:second-bound-boldpitildek-estimates}
\| (I- {\Pitildeze})v\|_{0,\e}
\lesssim \he^{s}\Norm{v}_{s,\e}.
\end{equation}
Collecting~\eqref{eqn:psin2} and~\eqref{eqn:second-bound-boldpitildek-estimates}
in~\eqref{eqn:difference-boldpik-boldpitildek},
we deduce
\[
\Norm{ \nabla v - \Pitildeoh \nabla v}_{0,\Omega}
\lesssim h^s \left(\Norm{v}_{s+1,\Omega}
                +\sum_{i=1}^N\Norm{v}_{s,\Gamma_i}\right).
\]
To end up with error estimates involving terms only in the domain~$\Omega$ and not on its boundary,
we use the Stein's extension operator
of Lemma~\ref{lemma:Stein}.
For any curve $\Gamma_i$ on the boundary of~$\partial\Omega$,
let $\mathcal C_i$ be a domain in $\Rbb^2$
with part of its boundary given by $\partial \mathcal{C}_i\in C^{k,1}$.
Then, applying the standard trace theorem on smooth domains
and the stability of the Stein's extension operator in~\eqref{eqn:Extension}, 
we obtain
\begin{equation}
\label{tracegammai}
\Norm{v}_{s,\Gamma_i} = \Norm{Ev}_{s,\Gamma_i}
\leq\|Ev\|_{s,\partial\mathcal{C}_i}
\lesssim \|Ev\|_{s+1,\mathcal{C}_i}
 \leq \|Ev\|_{s+1,\Rbb^d}
 \lesssim \Norm{v}_{s+1,\Omega}.
\end{equation}
The assertion follows combining the two estimates above.
\end{proof}

\section{The difference between continuous and discrete bilinear forms}
In many instances, we will need to estimate the difference between continuous and discrete bilinear forms. 
This estimation is addressed in the following preliminary lemma.

\begin{lemma}\label{lemma:L2consistency}
Let~$\E\in\calTh$, let~$\mu$ be a smooth function on~$\E$, and let~$p,~q$ denote
smooth scalar or vector-valued functions on~$\E$. 
Also, let~$Q_h^i,~i=1,~2,~3,~4$ be the operators defined on~$L^2(\E)$. We have 
\[\begin{aligned}
(\mu  p,q)_{\E}  - (\mu Q_h^1p, Q_h^2q)_{\E} 
&\lesssim \Norm{\mu p-Q_h^3(\mu p)}_{0,\E}\Norm{ q-Q_h^2 q}_{0,\E}+\Norm{p - Q_h^1 p}_{0,\E}  \Norm{q-Q_h^2q}_{0,\E}\\
&\quad+\Norm{p - Q_h^1 p}_{0,\E}\Norm{\mu q-Q_h^4(\mu q)}_{0,\E}+\SemiNorm{(Q_h^3(\mu p),q-Q_h^2 q)_{\E}} \\
&\quad+\SemiNorm{(p - Q_h^1 p,Q_h^4 (\mu q))_{\E}} .
\end{aligned}\]

\end{lemma}

\begin{proof}
    By adding and subtracting terms, we obtain
\[\begin{aligned}
     &(\mu  p,q)_{\E}   - (\mu Q_h^1p~Q_h^2q)_{\E} \\
     &= (\mu p, q-Q_h^2 q)_{\E}  + (\mu(p - Q_h^1 p), Q_h^2q)_{\E} \\
    & =(\mu p-Q_h^3(\mu p),q-Q_h^2 q)_{\E}  + (p - Q_h^1 p,\mu  Q_h^2q-Q_h^4 (\mu q))_{\E} \\
    &\quad+(Q_h^3(\mu p),q-Q_h^2 q)_{\E} +(p - Q_h^1 p,Q_h^4 (\mu q))_{\E} \\
    & =  (\mu p-Q_h^3(\mu p), q-Q_h^2 q)_{\E}  + (p - Q_h^1 p,\mu  Q_h^2q- \mu q)_{\E} \\
    &\quad+ (p - Q_h^1 p,\mu q-Q_h^4 (\mu q))_{\E} +(Q_h^3(\mu p),q-Q_h^2 q)_{\E} \\
    &\quad+(p - Q_h^1 p,Q_h^4 (\mu q)_{\E} .
\end{aligned}\]
and the result follows by the Cauchy–Schwarz inequality.
\end{proof}

By choosing different operators~$Q_h^i,~i=1$,~$2$,~$3$~and~$4$, the following result  immediately follows from the Lemma~\ref{lemma:L2consistency} and the~$L^2$ orthogonality of the project operator~$\Pik$.
\begin{lemma}\label{lemma:abcconsistency}
For~$\E\in\calTh$, for a generic~$w\in L^2(\E)$ (or~$\bm{w}\in[L^2(\E)]^d$), we define
  \[
   \calE_{\E}^{k-1}(w):=\Norm{w-\Pik w}_{0,\E},
\]
and 
\[\widetilde{{\calE}}_{\E}^{k-1}(\bm{w}):=\Norm{\bm{w}-\Pitildeo\bm{w}}_{0,\E}.\]
For any~$\uh,~\vh\in \VhE+\Pbb_k(\E)$ it holds
\[\begin{aligned}
    |\aE(\uh,\vh)- \ahE(\uh,\vh)|&\leq \EEk(a\nabla \uh)\Etildek(\nabla \vh)+C_a\Etildek(\nabla \uh)\Etildek(\nabla \vh)\\
    &\quad+\Etildek(\nabla \uh)\EEk(a\nabla \vh) +\SE((I-\Pitildenabla )\uh,(I-\Pitildenabla )\vh)\\
    &\quad+ \left|\int_{\E}\Pik(a \nabla \uh)\cdot(\nabla \vh-\Pitildeo \nabla \vh)~\dx\right|\\
    &\quad + \left|\int_\E(\nabla \uh- \Pitildeo \nabla \uh)\cdot \Pik(a \nabla \vh)~\dx\right|;\\
   | \bE(\uh,\vh)- \bhE(\uh,\vh)|
    &\leq
    \EEk(\bbm \uh)\Etildek(\nabla \vh)+C_{\bbm} \EEk(\uh)\Etildek(\nabla \vh)\\
    &\quad+\EEk(\uh)\EEk(\bbm\cdot\nabla \vh)
     +\left|\int_\E\Pik(\bbm \uh)\cdot(\nabla \vh -\Pitildeo\nabla \vh)~\dx\right|;\\
    |\cE(\uh,\vh)- \chE(\uh,\vh)|&\leq \EEk(c \uh)\EEk(\vh)+C_{c} \EEk(\uh)\EEk(\vh)+\EEk(\uh)\EEk(c\vh).
\end{aligned}\]  
\end{lemma}

\subsection{The error between the bilinear functions~$B(\upi,\vh)$ and $\Bh(\upi,\vh)$}

Firstly, we estimate from above a term measuring
the lack of polynomial consistency of the proposed method, i.e.,
the error between the bilinear functions~$B(\upi,\vh)$ and $\Bh(\upi,\vh)$. 
According to Lemma ~\ref{lemma:abcconsistency}, we also need the following results.

\begin{lemma}\label{lemma:pitildea1}
For any~$u\in H^{s+1}(\Omega)$,
$1 \le s \le k$ and for~$\upi$ as in~\eqref{eqn:L2-projection-bulk}, we have 
 \begin{equation}\label{eqn:pitildea1upi}
 \SemiNorm{\sum_{\E\in\calTh}\int_\E(\nabla \upi- \Pitildeo \nabla \upi)\cdot \Pik(a \nabla \vh)~\dx} \lesssim h^{s}\Norm{u}_{s+1,\Omega}\Norm{\vh}_{1,h}. \end{equation}

\end{lemma}
\begin{proof}

Applying the Cauchy–Schwarz inequality, we obtain
\begin{equation}\label{eqn:pitildea11}
  \begin{aligned}
    \int_\E(\nabla \upi- \Pitildeo \nabla \upi)\cdot \Pik(a \nabla \vh)~\dx \lesssim  \Etildek(\nabla \upi)\SemiNorm{\vh}_{1,\E}.
\end{aligned}  
\end{equation}

 The first factor in the right-hand side of~\eqref{eqn:pitildea11} can be easily bounded by~\eqref{eqn:Polynomial error estimates} and Lemmas~\ref{lemma:PTcont},
 \begin{equation}\label{eqn:pitildea12}
     \begin{aligned}
     &\Etildek(\nabla \upi) = \Norm{\nabla \upi-\Pitildeo\nabla \upi}_{0,\E}\\
 &\leq \Norm{\nabla \upi-\nabla u}_{0,\E}+\Norm{\nabla u-\Pitildeo\nabla u}_{0,\E}+\Norm{\Pitildeo(\nabla u-\nabla \upi)}_{0,\E}\\
  &\lesssim h^s\Norm{ u}_{s+1,\E}+\Norm{\nabla u-\Pitildeo\nabla u}_{0,\E}.
 \end{aligned}
 \end{equation}

The assertion follows summing over the elements, collecting the above inequalities, and using Lemma~\ref{lemma:v-pikt0v}.
\end{proof}

\begin{lemma}\label{lemma:pitildea2}
For any~$u\in H^{s+1}(\Omega)$,
$1 \le s \le k$ and for~$\upi$ as in~\eqref{eqn:L2-projection-bulk}, we have  
\begin{equation}\label{eqn:pitildea2}
   \begin{aligned}
   \SemiNorm{\sum_{\E\in\calTh} \int_{\E}\Pik a\nabla \upi \cdot(\nabla \vh - \Pitildeo\nabla \vh)~\dx}
   \lesssim h^{s}\Norm{u}_{s+1,\Omega}\SemiNorm{\vh}_{1,h}. 
\end{aligned} 
\end{equation}

\end{lemma}
\begin{proof}
For any element~$\E\in\calTh$~and~$u\in H^{s+1}(\Omega)$, we have
\[ \begin{aligned}
  \int_{\E}\Pik a\nabla \upi \cdot(\nabla \vh - \Pitildeo\nabla \vh)~\dx& =  \int_{\E}(\Pik a\nabla \upi-\Pik a\nabla u) \cdot(\nabla \vh - \Pitildeo\nabla \vh)~\dx\\
  &\quad+ \int_{\E}\Pik a\nabla u \cdot(\nabla \vh - \Pitildeo\nabla \vh)~\dx:= I_{1}+I_{2}.
\end{aligned} \]

As for the term~$I_{1}$, we use the stability of $\Pik$ in $L^2$ norm, Lemma~\ref{lemma:PTcont} and~\eqref{eqn:Polynomial error estimates}, then deduce
\[\begin{aligned}
    I_{1}&\leq \Norm{\Pik a\nabla \upi-\Pik a\nabla u}_{0,\E}\Norm{\nabla \vh - \Pitildeo\nabla \vh}_{0,\E}\lesssim h^s\Norm{u}_{s+1,\E}\SemiNorm{\vh}_{1,\E}.
\end{aligned}\]

As for the term~$I_{2}$, we denote $\bar{v}_h = \vh-\PipartialE\vh$ and use the definition of the operator $\Pitildeo$, then 
  \[
    \begin{aligned}
  I_2 &=    \int_{\E}(\Pik a\nabla u)\cdot(\nabla \bar{v}_h - \Pitildeo\nabla \bar{v}_h)\dx\\
 &= \sum_{e\in\EcalEc}\int_{\e} (I- \Pitildeze) (\nbfE\cdot \Pik a\nabla u) \bar{v}_h \drm s \\
 &\leq \sum_{\e\in\EcalE_c}
\Norm{(I- \Pitildeze) (\nbfE\cdot \Pik a\nabla u) }_{0,\e}\Norm{\bar{v}_h }_{0,\e}.
    \end{aligned}
    \]
    
Using the approximation of the operator $\Pitildeze$,
the Poincar\'e inequality,
and the trace inequality,
we infer
the inequality
 \begin{equation}\label{eqn:ar1}
\begin{aligned}
\hE^{\frac12}\Norm{(I- \Pitildeze) (\nbfE\cdot \Pik a\nabla u)}_{0,\e}&\lesssim  \hE^{s}\Norm{\nbfE\cdot \Pik a\nabla u}_{s-\frac12,\e}\lesssim  \hE^{s}\Norm{\Pik a\nabla u}_{s-\frac12,\e}\\
&\lesssim\hE^{s}\Norm{\Pik a\nabla u}_{s,\E}
\lesssim\hE^{s}\Norm{u}_{s+1,\E}.
\end{aligned}
\end{equation}

The assertion in~\eqref{eqn:pitildea2} follows collecting the above estimates,~\eqref{trace-inverse-vh} and summing over the elements.
\end{proof}

With similar arguments as Lemma~\ref{lemma:pitildea2} we have the following result.
\begin{lemma}\label{lemma:pitildeb}
For any~$u\in H^{s+1}(\Omega)$,
$1 \le s \le k$ and for~$\upi$ as in~\eqref{eqn:L2-projection-bulk}, we have    
\begin{equation}\label{eqn:pitildeb}
   \begin{aligned}
    \SemiNorm{\sum_{\E\in\calTh} \int_{\E}\Pik \bbm \upi\cdot(\nabla \vh - \Pitildeo\nabla \vh)~\dx}\lesssim h^{s}\Norm{u}_{s,\Omega}\SemiNorm{\vh}_{1,h}.
\end{aligned} 
\end{equation}
 
\end{lemma}

\begin{theorem}  
\label{lemma:Consistency}For any~$u\in H^{s+1}(\Omega)$,
$1 \le s \le k$ and for~$\upi$ as in~\eqref{eqn:L2-projection-bulk}, we have
\begin{equation}\label{eqn:Consistency}
|B(\upi,\vh)- \Bh(\upi,\vh)|\lesssim h^s\Norm{u}_{s+1,\Omega}\Norm{\vh}_{1,h}.
\end{equation}
\end{theorem}
\begin{proof}
From the definitions of~$B$ and~$B_h$ we have
\begin{equation}
\label{eqn:Iabc}
\begin{aligned}
    B(\upi,\vh)- \Bh(\upi,\vh)& = (a(\upi,\vh)- \ah(\upi,\vh))-(b(\upi,\vh)- \bh(\upi,\vh))\\
    &\quad+(c(\upi,\vh)- \ch(\upi,\vh))\\
    &:= I_a+I_{\bbm}+I_c.
\end{aligned}\end{equation}

As for the term $I_a$, we use (5.14) from \cite{Beirao-Brezzi-Marini-Russo:2016} to obtain
$
   \EEk(a\nabla \upi) \lesssim h^s\Norm{u}_{s+1,\E}.
$
From~\eqref{eqn:pitildea12}, we obtain
$$
  \Etildek(\nabla \upi)\lesssim h^s\Norm{u}_{s+1,\E}+\Norm{\nabla u-\Pitildeo\nabla u}_{0,\E}.
$$

And then, using the definition the of operator~$\Etildek$ and the stability property of the projection operator~$\Pitildeo$ in Lemma~\ref{lemma:PTcont}, the terms~$\Etildek(\nabla \vh)$ and~$\Etildek( a\nabla \vh)$ can be simply bounded by~$\Norm{\vh}_{1,\E}$. 
As for the stability $\SE$,
by using~\eqref{eqn:stability bounded property}, adding and subtracting~$u$, applying Lemma~\ref{lemma:PTcont} and~\eqref{eqn:Polynomial error estimates}, we deduce
\[\begin{aligned}
 \SE((I-\Pitildenabla )\upi,(I-\Pitildenabla )\vh)
    &\lesssim \SemiNorm{(I-\Pitildenabla )\upi}_{1,\E}\SemiNorm{(I-\Pitildenabla )\vh}_{1,\E}\\
    &\lesssim (\SemiNorm{u-\upi}_{1,\E}+\SemiNorm{(I-\Pitildenabla )u}_{1,\E})\SemiNorm{(I-\Pitildenabla )\vh}_{1,\E}\\
    &\lesssim (h^s\SemiNorm{u}_{s+1,\E}+\SemiNorm{u-\Pitildenabla u}_{1,\E})\SemiNorm{\vh}_{1,\E}.
\end{aligned}\]

Thus, the term~$I_a$ follows summing over the elements
and using Lemmas~\ref{lemma:v-pikv2},~\ref{lemma:v-pikt0v},~\ref{lemma:abcconsistency},~\ref{lemma:pitildea2} and~\ref{lemma:pitildea1}:
\begin{equation}\label{eqn:Ia}
   I_a\lesssim h^s\Norm{u}_{s+1,\Omega}\Norm{\vh}_{1,h}.
\end{equation}

With similar arguments, using Lemma~\ref{lemma:pitildeb} and Lemma~5.5 in \cite{Beirao-Brezzi-Marini-Russo:2016}, we obtain, for instance,
\begin{equation}
\label{eqn:Ibc}
\begin{aligned}
  I_{\bbm}\lesssim h^s\Norm{u}_{s,\Omega}\Norm{\vh}_{1,h}\quad\text{and}\quad 
   I_{c}\lesssim h^s\Norm{u}_{s,\Omega}\Norm{\vh}_{1,h}.
\end{aligned}\end{equation}

The proof followsby substituting equations~\eqref{eqn:Ia} and~\eqref{eqn:Ibc} into~\eqref{eqn:Iabc}.
\end{proof}
\subsection{The error between the bilinear functions~$B(\uh,\vI)$ and $\Bh(\uh,\vI)$}
As for the error estimate in~$L^2$ norm, we also need
the error estimate between the bilinear functions~$B(\uh,\vI)$ and $\Bh(\uh,\vI)$. 
The following results could be deduced from Lemma~\ref{lemma:abcconsistency}.

\begin{lemma}
\label{lemma:pitildea1ui}
 For any~$u\in H^{s+1}(\Omega)$,
$1 \le s \le k$,~$v\in H^2(\Omega)$,~$\uh \in \VhE+[\Pbb_{k}(\E
)]^2$, and for~$\vI $ as in~\eqref{eqn:vik}, we have
 \begin{equation}\label{eqn:pitildea1ui}
 \SemiNorm{\sum_{\E\in\calTh}\int_\E(\nabla \uh- \Pitildeo \nabla \uh)\cdot \Pik(a \nabla \vI)~\dx }\lesssim ( h\SemiNorm{u-\uh}_{1,h}+h^{s+1}\Norm{u}_{s+1,\Omega})\Norm{v}_{2,\Omega}.
 \end{equation}   
\end{lemma}
\begin{proof}
Set~$\bar{u}_h := \uh-\PipartialE\uh$. 
Obviously,
we have $ \nabla \bar u_h  = \nabla\uh$  
and arrive at
\[
  \begin{aligned}
&\sum_{\E\in\calTh}\int_\E(\nabla \uh- \Pitildeo \nabla \uh)\cdot \Pik(a \nabla \vI)~\dx\\
& =  \sum_{\E\in\calTh}\int_\E(\nabla \uh- \Pitildeo \nabla \uh)\cdot (\Pik(a \nabla \vI)-\Pi_{0}^{0,\E}(a\nabla v))~\dx\\
&\quad+\sum_{\E\in\calTh}\int_\E(\nabla \bar{u}_h- \Pitildeo \nabla \bar{u}_h)\cdot \Pi_{0}^{0,\E}(a\nabla v)~\dx\\
&: =I_1+I_2.
\end{aligned}  
\]

Like as~\eqref{eqn:pitildea11} and~\eqref{eqn:pitildea12}, using Lemma~\ref{lemma:DoFs-interpolant-estimates} and~\ref{lemma:v-pikt0v}, we have 
\[
  \begin{aligned}
    I_1 
    &\lesssim  (\SemiNorm{ \uh- u}_{1,h}+\Norm{\nabla u-\Pitildeoh\nabla u}_{0,\Omega})(\SemiNorm{\vI-v}_{1,h}+\Norm{\nabla v - \Pi_{k-1}^{0,h} \nabla v }_{0,\Omega}+\Norm{\nabla v - \Pi_0^{0,h} \nabla v }_{0,\Omega})\\
    &\lesssim ( h\SemiNorm{u-\uh}_{1,h}+h^{s+1}\Norm{u}_{s+1,\Omega})\Norm{v}_{2,\Omega}.
\end{aligned}  
\]

We handle the term~$I_2$ by employing the definition of the operator~$\Pitildeo$:
\[
    \begin{aligned}
  I_2  &=\sum_{\E\in\calT_h}\sum_{e\in\EcalEc}\int_{\e} (I- \Pitildeze) (\nbfE\cdot \Pi_0^{0,\E}( a\nabla v)) \bar{u}_h \drm s\\
 &\leq \sum_{\E\in\calT_h}\sum_{\e\in\EcalE_c} 
\Norm{(I- \Pitildeze) (\nbfE\cdot \Pi_0^{0,\E}( a\nabla v)) }_{0,\e}\Norm{(I- \Pitildeze)\bar{u}_h}_{0,\e}.
    \end{aligned}
    \]
    
We apply Lemma~\ref{lemma:proes},
note that~$\Pi_{0}^{0,\E} \nabla v$ is a constant vector,
recall that~$\nbfE$ is fixed as the curved boundary is fixed
whence~$\Norm{\nbfE}_{W^{1,\infty}(\e)} \lesssim 1$,
 and deduce 
    \[
\begin{aligned}
\|(I-\Pitildeze) (\nbfE\cdot  \Pi_0^{0,\E}( a\nabla v))\|_{0,\e}
&\lesssim \he\|\nbfE\cdot \Pi_0^{0,\E}( a\nabla u)\|_{1,\e} \lesssim  \he\|\Pi_0^{0,\E}( a\nabla v)\|_{1,\e}.
 \end{aligned}
\]

On the other hand, using again that~$\Pi_0^{0,\E}( a\nabla v)$
is a constant vector,
whence~$|\Pi_0^{0,\E}( a\nabla v)|_{1,\e}=0$.
By applying a polynomial trace inverse inequality and the approximation properties~$\Pi_0^{0,\E}$,
we can write
\begin{equation}\label{eqn:I-pitilde0enablapi}
\begin{aligned}
&\|(I-\Pitildeze) (\nbfE\cdot \nabla \Pi_0^{0,\E}( a\nabla v))\|_{0,\e}\lesssim  \he\|\Pi_0^{0,\E}( a\nabla v)\|_{0,\e}\\
&\lesssim  \he\|  a\nabla v-\Pi_0^{0,\E}( a\nabla v)\|_{0,\e}+\he\| a\nabla v\|_{0,\e}\\
&\lesssim  \hE^{\frac12}\|  a\nabla v-\Pi_0^{0,\E}( a\nabla v)\|_{0,\E}+\hE^{\frac32}|  a\nabla v-\Pi_0^{0,\E}( a\nabla v)|_{1,\E}+\hE|v|_{1,\e}\\
 &\lesssim \hE^{\frac32}| v|_{2,\E}+\hE|v|_{1,e},
 \end{aligned}
 \end{equation}
and using~\eqref{trace-inverse-vh}
\begin{equation}
\label{eqn:I-pitilde0ebar}
\begin{aligned}
& \|(I-\Pitildeze)\bar{u}_h\|_{0,\e}\
\leq\|(I-\Pitildeze)(\bar{u}-\bar{u}_h)\|_{0,\e}
+ \|(I-\Pitildeze)\bar{u}\|_{0,\e}
\\
&\leq \|\bar{u}-\bar{u}_h\|_{0,\e}
        +\he^{s} \|\bar{u}\|_{s,\e}
        \lesssim  (\hE^{\frac12}\SemiNorm{u-\uh}_{1,\E}
        +\hE^{s}\Norm{u}_{s,\e}).
 \end{aligned}
 \end{equation}
 
Similar to equation~\eqref{tracegammai}, we further obtain
\[\begin{aligned}
    \left(\sum_{\E\in\calT_h}\sum_{\e\in\EcalE_c}\|(I-\Pitildeze) (\nbfE\cdot \nabla \Pi_0^{0,\E}( a\nabla v))\|_{0,\e}^2\right)^{\frac12}&\lesssim h^{\frac32}| v|_{2,\Omega}+h(\sum_{i=1}^N\|v\|^2_{1,\Gamma_i})^{\frac12}\\&\lesssim h\Norm{ v}_{2,\Omega},
\end{aligned}\]
and 
\[\begin{aligned}
    \left(\sum_{\E\in\calT_h}\sum_{\e\in\EcalE_c}\Norm{(I- \Pitildeze)\bar{u}_h}_{0,\e} ^2\right)&\lesssim h^{\frac12}\SemiNorm{u-\uh}_{1,h}
        +h^{s}(\sum_{i=1}^N\Norm{u}_{s,\Gamma_i}^2)^{\frac12}\\&\lesssim h^{\frac12}\SemiNorm{u-\uh}_{1,h}
        +h^{s}\Norm{u}_{s+1,\Omega}.
\end{aligned}\]

This conclude the bound on the term~$I_2$:
\[I_2\lesssim h(\SemiNorm{u-\uh}_{1,h}
        +h^{s}\Norm{u}_{s+1,\Omega})\Norm{ v}_{2,\Omega}.\]

The assertion follows from the above estimates.        
\end{proof}
\begin{lemma}\label{lemma:pitildea2I}

For any~$u\in H^{s+1}(\Omega)$,
$1 \le s \le k$,~$v\in H^2(\Omega)$,~$\uh \in \VhE+[\Pbb_{k}(\E
)]^2$, and for~$\vI $ as in~\eqref{eqn:vik}, we have
\begin{equation}\label{eqn:pitildea2I}
   \begin{aligned}
   \SemiNorm{\sum_{\E\in\calTh} \int_{\E}\Pik a\nabla \uh \cdot(\nabla \vI - \Pitildeo\nabla \vI)~\dx}\lesssim( h\SemiNorm{u-\uh}_{1,h}+h^{s+1}\Norm{u}_{s+1,\Omega})\Norm{v}_{2,\Omega}.
\end{aligned} 
\end{equation}
\end{lemma}
\begin{proof}
For~$u\in H^{s+1}(\Omega)$, we have
\[ \begin{aligned}
  &\sum_{\E\in\calTh} \int_{\E}\Pik a\nabla \uh \cdot(\nabla \vI - \Pitildeo\nabla \vI)~\dx \\
  &= \sum_{\E\in\calTh} \int_{\E}(\Pik a\nabla \uh-\Pik a\nabla u) \cdot(\nabla \vI - \Pitildeo\nabla \vI)~\dx\\
  &\quad+\sum_{\E\in\calTh} \int_{\E}\Pik a\nabla u \cdot(\nabla \vI - \Pitildeo\nabla \vI)~\dx:= I_{1}+I_{2}.
\end{aligned} \]

We handle the term~$I_{1}$, employing the stability of the operator $\Pikh$, 
adding and subtracting~$v$, 
and using Lemma~\ref{lemma:PTcont}, then deduce
\[\begin{aligned}
    I_{1}&\leq \Norm{\Pikh a\nabla \uh-\Pikh a\nabla u}_{0,\Omega}\Norm{\nabla \vI - \Pitildeoh\nabla \vI}_{0,\Omega}\\
      &\lesssim \SemiNorm{u-\uh}_{1,h}(\SemiNorm{\vI-v}_{1,h}+\Norm{\nabla v-\Pitildeoh\nabla v}_{0,\Omega}).
\end{aligned}\]

As for the term~$I_{2}$, we denote $\bar{v}_I = \vI-\PipartialE\vI$ and use the definition of the operator $\Pitildeo$, then 
    \[
    \begin{aligned}
  I_2 &= \sum_{\E\in\calTh}   \int_{\E}(\Pik a\nabla u)\cdot(\nabla \bar{v}_I - \Pitildeo\nabla \bar{v}_I)\dx\\
 &=\sum_{\E\in\calTh} \sum_{e\in\EcalEc}\int_{\e} (I- \Pitildeze) (\nbfE\cdot \Pik a\nabla u) \bar{v}_I \drm s \\
 &\leq \sum_{\E\in\calTh}\sum_{\e\in\EcalE_c}
\Norm{(I- \Pitildeze) (\nbfE\cdot \Pik a\nabla u) }_{0,\e}\Norm{(I- \Pitildeze) \bar{v}_I }_{0,\e}.
    \end{aligned}
    \]
    
  For the case~$k=1$, as in the estimate of the term~$I_{2}$ in Lemma~\ref{lemma:pitildea1ui},
we deduce 
\[I_2\lesssim  h(\SemiNorm{v-\vI}_{1,h}
        +h\Norm{v}_{2,\Omega})\Norm{u}_{2,\Omega}\lesssim  h^2\Norm{u}_{2,\Omega}\Norm{v}_{2,\Omega}.\]
        
For the case~$k\geq2$, thanks to~\eqref{eqn:ar1}, we can write
 \[
\begin{aligned}
h^{\frac12}\Norm{(I- \Pitildeze) (\nbfE\cdot \Pik a\nabla u)}_{0,\e}& \lesssim h^s\Norm{u}_{s+1,\E},
\end{aligned}
\]

and from~\eqref{eqn:I-pitilde0ebar}, we infer
\[
\begin{aligned}
 \hE^{-\frac12}\|(I- \Pitildeze)\bar{v}_I\|_{0,\e}
 \lesssim |\vI-v|_{1,\E}+\hE\|v\|_{\frac32,\e}.
\end{aligned}
\]
    
Therefore, using Cauchy–Schwarz inequality and likes as~\eqref{tracegammai}, we further obtain
\[I_2\lesssim h^s\Norm{u}_{s+1,\Omega}(\SemiNorm{v-\vI}_{1,h}+h\Norm{v}_{2,\Omega}).\]

The assertion in~\eqref{eqn:pitildea2I} follows from combining the above estimates,
using the approximation properties of~$\vI$ in Lemma~\ref{lemma:DoFs-interpolant-estimates}, 
the approximation properties of~$\Pitildeoh$ in Lemma~\ref{lemma:v-pikt0v}.
\end{proof}

As in the estimate of  Lemma~\ref{lemma:pitildea2I} we have the following result:
\begin{lemma}\label{lemma:pitildeb2}
For any~$u\in H^{s+1}(\Omega)$,
$1 \le s \le k$,~$v\in H^2(\Omega)$,~$\uh \in \VhE+[\Pbb_{k}(\E
)]^2$, and for~$\vI $ as in~\eqref{eqn:vik}, we have
 \begin{equation}\label{eqn:pitildeb2}
   \begin{aligned}
   \SemiNorm{\sum_{\E\in\calTh} \int_{\E}\Pik \bbm \uh\cdot(\nabla \vI - \Pitildeo\nabla \vI)~\dx}\lesssim( h\Norm{u-\uh}_{0,\Omega}+h^{s+1}\Norm{u}_{s,\Omega})\Norm{v}_{2,\Omega}.
\end{aligned} 
\end{equation}
\end{lemma}
 Then, we proceed to introduce  the following consistency.
\begin{theorem}\label{lemma:ConsistencyL2}
For any~$u\in H^{s+1}(\Omega)$,
$1 \le s \le k$,~$v\in H^2(\Omega)$,~$\uh \in \VhE+[\Pbb_{k}(\E
)]^2$, and for~$\vI $ as in~\eqref{eqn:vik}, it holds
\begin{equation}\label{eqn:ConsistencyL2}
|B(\uh,\vI)- \Bh(\uh,\vI)|\lesssim( h\Norm{u-\uh}_{1,h}+h^{s+1}\Norm{u}_{s+1,\Omega})\Norm{v}_{2,\Omega}.
\end{equation}
\end{theorem}
\begin{proof}
 As in the estimate of Theorem \ref{lemma:Consistency}, we deduce~\eqref{eqn:ConsistencyL2} with adding and subtracting~$u$, using Lemma~\ref{lemma:pitildea1ui}, \ref{lemma:pitildea2I} and~\ref{lemma:pitildeb2}.
\end{proof}
\subsection{The error between the bilinear functions~$B(\uh,\vI)$ and $\Bh(\uh,\vI)$ for stability analysis }
To derive a uniqueness and existence
result for~$\uh$, we also need
the error between the bilinear functions~$B(\uh,\vI)$ and $\Bh(\uh,\vI)$. And then the following results could be deduced from Lemma~\ref{lemma:abcconsistency}.

\begin{lemma}\label{lemma:adiffah1}
For any~$v\in H^2(\Omega)$,~$\uh \in \VhE+[\Pbb_{k}(\E
)]^2$, and for~$\vI $ as in~\eqref{eqn:vik}, we have
\begin{equation}\label{eqn:adiffah1}
   \begin{aligned}
   \SemiNorm{\sum_{\E\in\calTh} \int_{\E}\Pik a\nabla \uh \cdot(\nabla \vI - \Pitildeo\nabla \vI)~\dx}\lesssim h\Norm{\uh}_{1,h}\Norm{v}_{2,\Omega}.
\end{aligned} 
\end{equation}
\end{lemma}
\begin{proof}
Following by the Cauchy–Schwarz inequality, Lemma~\ref{lemma:PTcont} and \eqref{eqn:pitildea12}, we obtain 
    \[
    \begin{aligned}
  &\int_{\E}\Pik a\nabla \uh \cdot(\nabla \vI - \Pitildeo\nabla \vI)~\dx\lesssim \Norm{\uh}_{1,\E} \calE_{\E}^{k-1}(\vI)\\
  &\lesssim \Norm{\uh}_{1,\E}(h\Norm{ v}_{2,\E}+\Norm{\nabla v-\Pitildeo\nabla v}_{0,\E}).
    \end{aligned}
    \]

The assertion follows from summing over the elements and applying Lemma~\ref{lemma:v-pikt0v}.
  
\end{proof}

\begin{lemma}
 For any~$v\in H^2(\Omega)$,~$\uh \in \VhE+[\Pbb_{k}(\E
)]^2$, and for~$\vI $ as in~\eqref{eqn:vik}, we have

\begin{equation}
 \SemiNorm{\sum_{\E\in\calTh}\int_\E(\nabla \uh- \Pitildeo \nabla \uh)\cdot \Pik(a \nabla \vI)~\dx} \lesssim  h\Norm{\uh}_{1,h}\Norm{v}_{2,\Omega}. \end{equation}    
\end{lemma}
\begin{proof}
Following the proof of Lemma~\ref{lemma:pitildea1ui} and , we set~$\bar{u}_h := \uh-\PipartialE\uh$,
and arrive at 
\[
  \begin{aligned}
&\sum_{\E\in\calTh}\int_\E(\nabla \uh- \Pitildeo \nabla \uh)\cdot \Pik(a \nabla \vI)~\dx\\
& =  \sum_{\E\in\calTh}\int_\E(\nabla \uh- \Pitildeo \nabla \uh)\cdot (\Pik(a \nabla \vI)-\Pi_{0}^{0,\E}(a\nabla v))~\dx\\
&\quad+\sum_{\E\in\calTh}\int_\E(\nabla \bar{u}_h- \Pitildeo \nabla \bar{u}_h)\cdot \Pi_{0}^{0,\E}(a\nabla v)~\dx\\
&: =I_1+I_2.
\end{aligned}  
\]

As for the first term, using Lemma~\ref{lemma:PTcont} and~\ref{lemma:v-pikt0v}, we have 
\[
  \begin{aligned}
    I_1 \lesssim  h\Norm{ \uh}_{1,h}\Norm{v}_{2,\Omega}.
\end{aligned}  
\]

We handle the term~$I_2$ with using the definition of the operator~$\Pitildeo$:
\[
    \begin{aligned}
  I_2 \leq \sum_{\E\in\calT_h}\sum_{\e\in\EcalE_c} 
\Norm{(I- \Pitildeze) (\nbfE\cdot \Pi_0^{0,\E}( a\nabla v)) }_{0,\e}\Norm{(I- \Pitildeze)\bar{u}_h}_{0,\e}.
    \end{aligned}
    \]
    
We apply \eqref{eqn:I-pitilde0enablapi}, and deduce 
    
\[
\begin{aligned}
\|(I-\Pitildeze) (\nbfE\cdot \nabla \Pi_0^{0,\E}( a\nabla v))\|_{0,\e}
 &\lesssim \hE^{\frac32}| v|_{2,\E}+\hE|v|_{1,e},
 \end{aligned}
\]

and use~\eqref{trace-inverse-vh}, it holds
\[
\begin{aligned}
 \|(I-\Pitildeze)\bar{u}_h\|_{0,\e}\
\leq\|\bar{u}_h\|_{0,\e}
        \lesssim  h^{\frac12}|\uh|_{1,\E}.
 \end{aligned}
\]

Suming over the elements, this concludes the bound on the term~$I_2$:
\[I_2\lesssim h\SemiNorm{\uh}_{1,h}\Norm{ v}_{2,\Omega}.\]
        
Combining the above estimates, we obtain the assertion in \eqref{eqn:adiffah1}.
\end{proof}

With similar arguments as Lemma~\ref{lemma:adiffah1} and Theorem~\ref{lemma:Consistency} we have the following results
\begin{lemma}
For any~$v\in H^2(\Omega)$,~$\uh \in \VhE+[\Pbb_{k}(\E
)]^2$, and for~$\vI $ as in~\eqref{eqn:vik}, we have 
\begin{equation}
   \begin{aligned}
    \SemiNorm{\sum_{\E\in\calTh} \int_{\E}\Pik \bbm \uh\cdot(\nabla \vI - \Pitildeo\nabla \vI)~\dx}\lesssim h\Norm{\uh}_{1,h}\Norm{v}_{2,\Omega}.
\end{aligned} 
\end{equation}
 
\end{lemma}

\begin{theorem}\label{lemma:BminusBhunique} For any~$v\in H^2(\Omega)$,~$\uh \in \VhE+[\Pbb_{k}(\E
)]^2$, and for~$\vI $ as in~\eqref{eqn:vik}, we have 
\begin{equation}
|B(\uh,\vI)- \Bh(\uh,\vI)|\lesssim h\Norm{\uh}_{1,h}\Norm{v}_{2,\Omega}\qquad.
\end{equation}
\end{theorem}

\section{Stability analysis }\label{section:stability}

Assume that~$\uh$ is an approximation for the problem \eqref{eqn:primal problem} arising from \eqref{eqn:VEM} by using the nonconforming virtual element space~$\Vh$. 
The objective of this section is to establish results for the uniqueness and existence of~$\uh$. 
First of all, we will present the stability estimates for~$\ahE$ as follows.
\begin{lemma}\label{prop:Xnew}
Given an element~$\E$ and the discrete bilinear form~$\ahE(\cdot,\cdot)$ based on the stabilization described in~\eqref{eqn:dofi-dofi}, we have the following stability estimate

\[
\SemiNorm{\vh}^2_{1,\E}
\lesssim \ahE(\vh,\vh)
\lesssim \SemiNorm{\vh}^2_{1,\E} 
\qquad\qquad \forall \vh \in \VhE .
\]
\end{lemma}
\begin{proof}

Using~\eqref{eqn:stability bounded property} and the first on in Lemma~\ref{lemma:PTcont}, we have
\begin{equation}\label{eqn:Sbound}
  \begin{aligned}
    \SE((I-\Pitilde_k^{\nabla,\E} )\uh,(I-\Pitilde_k^{\nabla,\E} )\vh)
    &\leq \alpha^*\SemiNorm{(I-\Pitilde_k^{\nabla,\E} )\uh}_{1,\E}\SemiNorm{(I-\Pitilde_k^{\nabla,\E} )\vh}_{1,\E}\\
     &\leq \alpha^*\SemiNorm{\uh}_{1,\E}\SemiNorm{\vh}_{1,\E}.
\end{aligned}  
\end{equation}

Let~$a_{\max} = \Norm{a}_{L^{\infty}(\Omega)}$ be the~$L^{\infty}$ norm of the coefficient~$a$, with Lemma~\ref{lemma:PTcont}, we can deduce
\[\ahE(\uh,\vh)\lesssim(a_{\max}+\alpha^*)\SemiNorm{\uh}_{1,\E}\SemiNorm{\vh}_{1,\E}.\]

Next, we focus on the lower bound.
By the properties of the operators~$\Pitilde_k^{\nabla,\E}$  and~$\Pitildeo $, we can derive
\[
\begin{aligned}
\Norm{\nabla\Pitilde_k^{\nabla,\E}\vh}_{0,\E}^2
 &= -\int_\E\vh~ \nabla\cdot(\nabla\Pitilde_k^{\nabla,\E}\vh)~\dx
 +\sum_{\e\in\EcalE}\int_\e\vh~ \Pitilde_{k-1}^{0,e}
(\nbfE\cdot\nabla \Pitilde_k^{\nabla,\E}\vh)~\ds\\
& =\int_\E\Pitildeo\nabla \vh\cdot\nabla\Pitildenabla\vh~\dx
\leq \Norm{\Pitildeo\nabla \vh}_{0,\E}\Norm{\nabla\Pitilde_k^{\nabla,\E}\vh}_{0,\E},
\end{aligned}
\]
giving immediately
\[\Norm{\nabla\Pitildenabla\vh}_{0,\E}
\leq \Norm{\Pitildeo\nabla \vh}_{0,\E}.\]

Following the proof of Proposition 4 in \cite{Beirao-Liu-Mascotto-Russo:2023}, we have 
\[\SemiNorm{\vh}_{1,\E}^2\lesssim
\Norm{\nabla\Pitildenabla\vh}_{0,\E}^2+ \SE((I-\Pitilde_k^{\nabla,\E} )\uh,(I-\Pitilde_k^{\nabla,\E} )\vh).\]

Combining the inequalities above, we obtain the desired result.
\end{proof}

We also have the following continuity of the global bilinear term~$B_h$.
\begin{lemma}
The bilinear form~$B_h$ is continuous in~$\Vh\times\Vh$, that is,
\begin{equation}\label{eqn:Continuity}
    B_h(u_h,v_h)\lesssim\Norm{u_h}_{1,h}\Norm{v_h}_{1,h}.
\end{equation}
\end{lemma}
\begin{proof}
    According to the stability of the operators~$\Pitildeo$ and~$\Pik$ in~$L^2$ norm, 
    the continuity of~$\bh$ and~$\ch$ is obvious, 
    and actually holds on the whole~$H^1_0(\Omega)$ space. 
    Let~$b_{\max} = \Norm{\bbm}_{L^{\infty}(\Omega)}$ and~$c_{\max} = \Norm{c}_{L^{\infty}(\Omega)}$ be the~$L^{\infty}$ norms of the coefficients~$\bbm$ and~$c$, respectively. 
    Therefore,
    we have
\begin{equation}   
\label{eqn:bccontinuity}
\begin{aligned}
| \bh(\uh,\vh) | &\lesssim b_{\max} \Norm{ \uh}_{0,\Omega}\SemiNorm{\vh}_{1,h};\\
| \ch(\uh,\vh) | &\lesssim c_{\max} \Norm{\uh}_{0,\Omega}\Norm{\vh}_{0,\Omega}.
\end{aligned}
\end{equation}
Thus the result follows Lemma~\ref{prop:Xnew} and~\eqref{eqn:bccontinuity}.
\end{proof}
And then we derive the following analogy of Gåding's inequality.
\begin{lemma}
 There exist positive constants~$\beta $ and~$\lambda$ satisfying
\begin{equation}
\label{eqn:Gading's inequality}
\Bh(\vh,\vh)+\beta \Norm{\vh}_{0,\Omega}^2\geq \lambda\|\vh\|_{1,h}^2.
\end{equation}
for all $\vh\in\Vh$.
\end{lemma}
\begin{proof} Obviously, the result follows from~\eqref{eqn:bilinear forms} and~\eqref{eqn:bccontinuity} that exists a constant constants~$\beta $ and~$\lambda$ such that 
\[\begin{aligned}
\Bh(\vh,\vh)+\beta \Norm{\vh}_{0,\Omega}^2
&\geq a_0 \SemiNorm{ \vh}_{1,h}^2-b_{\max} \SemiNorm{ \vh}_{1,h}\Norm{\vh}_{0,\Omega}+(\beta -c_{\max})\Norm{\vh}_{0,\Omega}^2\\
&\geq\lambda\|\vh\|_{1,h}^2.
\end{aligned}\]
where~ $\epsilon>0$.
\end{proof}
We are now in a position to establish the uniqueness and existence of the solution for the nonconforming virtual element method \eqref{eqn:VEM}. 
It is sufficient to prove that the solution is unique. 
To this end, let $ e\in \Vho$  be a discrete solution satisfying
\begin{equation}\label{eqn:unique}
\Bh(e,\vh) =0\qquad \qquad\forall \vh\in \Vho.
\end{equation}
The goal is to show that~$e\equiv0$ using a duality approach similar to the method used by Schatz \cite{Schatz:1974} for the standard Galerkin finite element methods.

Next, we introduce a bilinear form to account for the nonconformity of the method in the error analysis
$\calNha: H^{\frac32+\varepsilon}(\Omega) \times H^{1,nc}(\calTh,k)$, defined as
\begin{equation}\label{eqn:nc-term}
\calNha(u,\vh)
:= \sum_{\E \in \calTh}
    \int_{\partial \E} (\nbfE \cdot a\nabla u) \vh~\mathrm{d}s
=  \sum_{\e \in \calEh} \int_\e a\nabla u \cdot \jump{\vh}~\mathrm{d}s.
\end{equation}

In the following result, we can obtain the estimate of the term related to the nonconformity of the scheme \cite{Beirao-Liu-Mascotto-Russo:2023}.
\begin{lemma} \label{lemma:nc-term}
Let $u$ belong to~$H^{s+1}(\Omega)$, $1/2< s \le k$.
Then, for all~$\vh$ in~$\Vh$, we have
\[
|\calNha(u,\vh)|
\lesssim h^{s} \Norm{u}_{s+1,\Omega} \SemiNorm{v_h}_{1,h}.
\]
where~$\calNha(u,\vh)$ is defined in \eqref{eqn:nc-term}.
\end{lemma}

\begin{lemma}\label{lemma:uniqueL2}
Let~$e\in \Vho$ be a discrete solution satisfying \eqref{eqn:unique}. 
Assume that the dual of \eqref{eqn:primal problem} with homogeneous Dirichlet boundary condition has the ~$H^2$ regularity. Then
\[\Norm{e}_{0,\Omega}\lesssim h\SemiNorm{e}_{1,h},\]
provided that the meshsize~$h$ is sufficient small.
\end{lemma} 
\begin{proof}
Consider the following dual problem. Find~$\phi\in H^1(\Omega)$ such that
\begin{equation}\label{eqn:dual problem}
 \left\{
 \begin{aligned}
 & -\nabla\cdot (a\nabla \phi)-\bbm\cdot\nabla \phi+c \phi = e\quad \text{in}~\Omega,\\
 &\phi=0\quad\text{on}~\partial\Omega.
 \end{aligned}\right.
\end{equation}

The assumption of~$H^2$ regularity implies that~$\phi\in H^2(\Omega)$ 
and there exists a constant~$C$ such that
\begin{equation}\label{eqn:regularity}\Norm{\phi}_{2,\Omega}\lesssim \Norm{e}_{0,\Omega}.\end{equation}

Testing \eqref{eqn:dual problem} against~$e$ and then using an integration by part lead to
\[
\begin{aligned}
\Norm{e}_{0,\Omega}^2&=\int_{\Omega}(-\nabla \cdot a\nabla \phi-\bbm\cdot\nabla \phi+c\phi)e~\dx\\
&=\sum_{\E\in\calTh}\big{(}\int_{\E} a\nabla \phi\cdot\nabla e~\dx-\int_{\E} \bbm\cdot\nabla \phi e~\dx+\int_{\E}c\phi e~\dx\big{)}-\sum_{\E\in\calTh}\int_{{\partial \E}}(a\nabla \phi)\cdot\nbf_{\E} e~\drm s\\
&=B(e,\phi)-\calNha(\phi,e)=B(e,\phi-\phi_I)+B(e,\phi_I)-\calNha(\phi,e).
\end{aligned}
\]

We now estimate the three terms above. 
The estimate for the first one follows from the continuity of~$B(\cdot,\cdot)$
together with the approximation properties~\eqref{eqn:Approximation} of~$\phi_I$ and ~$\phi$
\[
\begin{aligned}
|B(e,\phi-\phi_I)|\lesssim \Norm{e}_{1,h}\Norm{\phi-\phi_I}_{1,h}\lesssim h\Norm{e}_{1,h}\Norm{\phi}_{2,\Omega}.
\end{aligned}\]

To estimate the second term, we can get~$B_h(e,\phi_I)=0$ due to \eqref{eqn:unique}, then use Theorem \ref{lemma:BminusBhunique}, we have 
$$|B(e,\phi_I)| =|B(e,\phi_I) -  B_h(e,\phi_I)|\leq h\Norm{e}_{1,h}\Norm{\phi}_{2,\Omega}.$$

Last term is readily estimated by means of Lemma \ref{lemma:nc-term} with~$ k=s=1$, giving
\[
|\calNha(\phi,e)|\lesssim  h\Norm{\phi}_{2,\Omega}\SemiNorm{e}_{1,h}.
\]

Using the~$H^2$-regularity assumption \eqref{eqn:regularity}, we arrive 
\[\Norm{e}_{0,\Omega}\lesssim h(\SemiNorm{e}_{1,h}+\Norm{e}_{0,\Omega}).\]

Thus, when~$h$ is sufficiently small, one would obtain the desired estimate. 
This completes the proof.
\end{proof}

\begin{theorem}
Assume that the dual of \eqref{eqn:primal problem} with homogeneous Dirichlet boundary condition has~$H^2$-regularity. 
The nonconforming virtual element method defined in \eqref{eqn:VEM} has a unique solution in the finite element space~$\Vh$ if the meshsize~$h$ is sufficiently small.
\end{theorem}
\begin{proof}
Observe that uniqueness is equivalent to existence for the solution of \eqref{eqn:VEM} since the number of unknowns matches the number of equations. 
To prove a uniqueness, let~$\uh^1$ and~$\uh^2$ be two solutions of \eqref{eqn:VEM}. 
By defining~$e = \uh^1- \uh^2$~we see that \eqref{eqn:unique} is satisfied. 
Now we have from Gårding’s inequality~\eqref{eqn:Gading's inequality} that
\[\Bh(e,e)+\beta \Norm{e}^2_{0,\Omega}\geq \lambda\|e\|_{1,h}^2.\]

Thus, it follows from the estimate of Lemma \ref{lemma:uniqueL2} that
\[\beta \Norm{e}^2_{0,\Omega}\lesssim \beta  h^2\SemiNorm{e}_{1,h}^2.\]

 Now choose~$h$ small enough so that~$\beta  h^2\lesssim\frac{\lambda}{2}$.
 Thus,
\[\Bh(e,e)\geq \frac{\lambda}{2}\|e\|_{1,h}^2.\]
and~$\|e\|_{1,h}=0$. This shows that~$e = 0$ and consequently,~$\uh^1=\uh^2$. 

\end{proof}

\section{Convergence analysis}\label{section:Convergence analysis}
The goal of this section is to derive error estimates for the virtual element method \eqref{eqn:VEM}.
We will follow the standard approach in the error analysis: 
investigating the difference between the virtual element approximation~$\uh$ with the exact solution~$u$ through an error equation;
and using a duality argument to analyze the error in the~$L^2$ norm.

To begin with the derivation of an error equation between the virtual element approximation~$\uh$ 
and the exact solution~$u$,
We test \eqref{eqn:primal problem} against~$\vh\in \Vh$. This leads to
\begin{equation}\label{eqn:nonconforming term}
\begin{aligned}
F(v)&=\sum_{\E\in\calTh}\int_{\E} ( -\nabla\cdot (a\nabla u)+\nabla\cdot(\bbm u)+c u)v~\dx\\
&=\sum_{\E\in\calTh}\int_{\E} ( a\nabla u)\cdot\nabla v~\dx-\sum_{\E\in\calTh}\int_{{\partial \E}}(a\nabla v)\cdot\nbf_{\E} v~\ds\\
&\quad-\sum_{\E\in\calTh}\int_{\E}u(\bbm\cdot\nabla v)~\dx+\sum_{\E\in\calTh}\int_{{\partial \E}} \bbm u\cdot\nbf_{\E} v~\ds\\
&\quad+\sum_{\E\in\calTh}\int_{\E}cuv~\dx\\
& := B(u,v)-\calNha(u,v)+\calNhb(u,v),
\end{aligned}
\end{equation}
where~$\calNha$ is defined as~\eqref{eqn:nc-term} and~$\calNhb: H^{\frac32+\varepsilon}(\Omega) \times H^{1,nc}(\calTh,k)$
 is defined as
\begin{equation}\label{eqn:nc-termb}
\calNhb(u,\vh)
:= \sum_{\E\in\calTh}\int_{{\partial \E}} \bbm u\cdot\nbf_{\E} v \ds
=  \sum_{\e \in \calEh} \int_\e \bbm u \cdot \jump{\vh}~\ds,
\end{equation}
then, we denote
\begin{equation}\label{eqn:nc-termall}
\calNh(u,\vh)
:= \calNha(u,\vh)-\calNhb(u,\vh).
\end{equation}

Next, we begin by recalling the discretization error on the right-hand side and the approximiation property of~$\Pik$.

\begin{lemma}\label{esf}
Let the right-hand side~$f$ of~\eqref{eqn:primal problem}
belong to~$H^{s}(\Omega)$, $1\le s \le k$,
and the discrete right-hand sides~$F$ and $F_h$ be as in~\eqref{eqn:bilinear terms} and~\eqref{eqn:fk1}.
Then, for any~$\vh$ in~$\Vh$, we have
\[
| F(\vh)-F_h(\vh)|\lesssim 
h^{s} \SemiNorm{f}_{s,\Omega}\Norm{\vh}_{1,h}.
\]
\end{lemma}

In the following result,
we cope with the estimate for the term related to the nonconformity of the scheme.

\begin{lemma}\label{lemma:nc-termall}
    Let $u$ belong to~$H^{s+1}(\Omega)$, $3/2< s \le k+1$.
Then, for all~$\vh$ in~$\Vh$, we have
\[
|\calNh(u,\vh)|
\lesssim h^{s} \Norm{u}_{s+1,\Omega} \SemiNorm{v_h}_{1,h}.
\]
where~$\calNh(u,\vh)$ is defined in \eqref{eqn:nc-termall}.
\end{lemma}
\begin{proof}
    Following the Lemma \ref{lemma:nc-term}, 
    we have shown the estimate for~$\calNha$. 
   We now turn to discussing the second term of~$\calNh$. 
    The proof follows along the same line as the one for the classical nonconforming methods. 
    We briefly report it here for the sake of completeness.
    From the definition of the space ~$H^{1,nc}(\calTh;k)$ and applying the Cauchy–Schwarz inequality, we obtain
    \[\begin{aligned}
        |\calNhb(u,\vh)|
        &= |\sum_{\e \in \calEh} \int_\e \bbm u \cdot \jump{\vh}~\mathrm{d}s|\\
        & = |\sum_{\e \in \calEh} \int_\e (\bbm u  - \Pitildeze(\bbm u))\cdot \jump{\vh}~\mathrm{d}s| \\
        & = |\sum_{\e \in \calEh} \int_\e (\bbm u  - \Pitildeze(\bbm u))\cdot (\jump{\vh}-[\![ \Pi_0^{0,e} \vh ]\!])~\mathrm{d}s|, \\
    \end{aligned}\]
Using now Lemma~\ref{lemma:proes} we have for each~$\e=\partial\E^+\cap\partial\E^-$, we get
\[
\|\bbm u  - \Pitildeze(\bbm u)\|_{0,\e}
\lesssim  h^k\| u\|_{k,\e}
    \lesssim  h^k\| u\|_{k+1,\E^+\cup\E^-},
\]
and
\[
\begin{aligned}
\|\jump{\vh}-  [\![ \Pi_0^{0,\e} \vh ]\!] \|_{0,\e}
& \lesssim  h^{-\frac12}\| \vh-  \Pi_0^{0,\e} \vh \|_{0,\E^+\cup\E^-}
+ h^{\frac12}| \vh-\Pi_0^{0,\e} \vh |_{1,\E^+\cup\E^-}
\lesssim  h^{\frac12}|\vh|_{1,\E^+\cup\E^-}.
\end{aligned}
\]
An analogous estimate is valid for boundary edges.
Summing over all elements, the assertion follows.
\end{proof}

\subsection{An estimate in a discrete~$H^1$-norm}\label{subsection:An estimate in a discreteH1-norm}

We are now in a position to state the abstract error analysis in the energy norm for method~\eqref{eqn:VEM}
and derive the corresponding optimal error estimate.

 \begin{lemma}\label{lemma:H1estimate} 
Let $u$ and~$\uh$ be the solutions to~\eqref{eqn:primal problem} and~\eqref{eqn:VEM}, respectively.
Then, for every $\uI$ in~$\Vh$, we have

\begin{equation}\label{eqn:Strang}
\begin{aligned}
\Norm{u-\uh}_{1,h}
 &\lesssim(\Norm{u-\uI}_{1,h}+\Norm{u-\upi}_{1,h}+\sup_{\vh\in \Vh}
          \left(\frac{| F(\vh)-F_h(\vh)|}{\Norm{\vh}_{1,h}}\right.\\
         &\quad \left.+\frac{|\calNh(u,\vh)|}{\Norm{\vh}_{1,h}} 
          +\frac{|B(\upi,\vh)-\Bh(\upi,\vh)|}{\Norm{\vh}_{1,h}}+\frac{|B(u-\upi,\vh)|}{\Norm{\vh}_{1,h}}\right)\\
          &\quad + \Norm{u-\uh}_{0,\Omega}).
\end{aligned}
\end{equation}
provided that the meshsize~$h$ is sufficiently small. Furthermore, if~$u \in H^{s+1}(\Omega)$ and~$f \in H^{s-1}(\Omega)$, with $1 \le s \le k$, we also have
\begin{equation} \label{eqn:error-estimates-H1}
\Norm{u-\uh}_{1,h}
\lesssim h^{s}(\Norm{u}_{s+1,\Omega}
+\SemiNorm{f}_{s-1,\Omega})+\Norm{u-\uh}_{0,\Omega}.
\end{equation}
 \end{lemma}
 
 \begin{proof}
We first prove the Strang-type estimate~\eqref{eqn:Strang}. We split~$u-\uh$ as~$(u-\uI )+(\uI-\uh)$,
use the triangle inequality, and get
\[
\Norm{u-\uh}_{1,h}
\leq \Norm{u-\uI}_{1,h}+\Norm{\uI-\uh}_{1,h}.
\]
Set $e_h := \uh-\uI$. Following the lines of \cite[Theorem~$4.3$]{AyusodeDios-Lipnikov-Manzini:2016} and recalling \eqref{eqn:nonconforming term}, \eqref{eqn:Continuity}
\[
\begin{aligned}
\Bh(e_h,e_h)&= \Bh(\uh,e_h)-\Bh(\uI,e_h)\\
& = F_h(\eh)-F(\eh) - \calNh(u,e_h)
-\Bh(\uI-\upi,e_h)\\
&\quad +(\Bh(\upi,e_h)-B(\upi,e_h))+B(u-\upi,e_h) \\
&\leq  I_h\Norm{e_h}_{1,h}.
\end{aligned}
\]
where 
\[
\begin{aligned}
I_h&= \Norm{u-\uI}_{1,h}+\Norm{u-\upi}_{1,h}+\sup_{\vh\in \Vh}
          \left(\frac{| F(\vh)-F_h(\vh)|}{\Norm{\vh}_{1,h}}\right.\\
         &\quad \left.+\frac{|\calNh(u,\vh)|}{\Norm{\vh}_{1,h}} 
          +\frac{|B(\upi,\vh)-\Bh(\upi,\vh)|}{\Norm{\vh}_{1,h}}+\frac{|B(u-\upi,\vh)|}{\Norm{\vh}_{1,h}}\right).
\end{aligned}
\]

Next, we use Gårding's inequality \eqref{eqn:Gading's inequality} to obtain
\[
\begin{aligned}
\lambda\Norm{e_h}_{1,h}^2
&\leq I_h\Norm{e_h}_{1,h}+\beta \Norm{e_h}^2_{0,\Omega}\\
&\leq\frac{\lambda}{2}\Norm{e_h}_{1,h}^2+\frac{2}{\lambda}I_h^2+\beta \Norm{e_h}^2_{0,\Omega}.
\end{aligned}\]
which implies the desired estimate \eqref{eqn:Strang}. The error estimates~\eqref{eqn:error-estimates-H1} finally follow bounding the terms
on the right-hand side of~\eqref{eqn:Strang} by means of
Lemmas~\ref{lemma:DoFs-interpolant-estimates}, \ref{esf}, \ref{lemma:nc-termall},
and~Theorem \ref{lemma:Consistency}.
\end{proof}
\subsection{An estimate in~$L^2$ norm}
We are now in a position to present the abstract error analysis in the~$L^2$ norm for the method.

\begin{lemma}
\label{lemma:L2estimate}
Assume that the dual of the problem \eqref{eqn:primal problem} has the~$H^{2}$ regularity. 
Let~$u\in H^{s+1}(\Omega)$ be the solution \eqref{eqn:primal problem},~$1\leq s\leq k$,
and~$\uh$ be a virtual element approximation of~$u$ arising from \eqref{eqn:VEM} by using either the discrete space~$\Vh$. 
Then
\[\Norm{u-\uh}_{0,\Omega}\lesssim h^{s+1}(\SemiNorm{f}_{s-1,\Omega}+\Norm{u}_{s+1,\Omega})+h\Norm{u-\uh}_{1,h}.\]
provided that the meshsize~$h$ is sufficiently small.
\end{lemma}

\begin{proof}
Consider the dual problem of \eqref{eqn:primal problem} which seeks $w\in H_0^1(\Omega)$.
\begin{equation}\label{eqn:dual problemL2}
 -\nabla\cdot (a\nabla \phi)-\bbm\cdot\nabla \phi+c \phi = u-u_h\quad \text{in}~\Omega.
\end{equation}

The assumed $H^{2}$ regularity for the dual problem implies

\begin{equation}
    \Norm{\phi-\phi_I}_{1,h}\lesssim h\Norm{\phi}_{2,\Omega}\lesssim h\Norm{u-u_h}_{0,\Omega}.
\end{equation}

Similarly argument with Lemma \ref{lemma:uniqueL2} and testing \eqref{eqn:dual problemL2} against~$e_h$ gives
\[
\begin{aligned}
\Norm{u-u_h}_{0,\Omega}^2
&=B(u-u_h,\phi-\phi_I)+B(u-u_h,\phi_I)-\calNha(\phi,u-u_h).
\end{aligned}
\]

We now estimate the three terms above.
The estimate for the first one follows from the continuity of~$B(\cdot,\cdot)$ together with the approximation property~\eqref{eqn:Approximation} of ~$\phi_I$ and the elliptic regularity theory:
\[
\begin{aligned}
|B(u-u_h,\phi-\phi_I)|\lesssim \Norm{u-u_h}_{1,h}\Norm{\phi-\phi_I}_{1,h}\lesssim h\Norm{u-u_h}_{1,h}\Norm{u-u_h}_{0,\Omega}.
\end{aligned}\]

Last term is readily estimated by means of Lemma~\ref{lemma:nc-term} with~$k=s=1$ (since obviously~$u-\uh \in H^{1,nc}(\calTh;1)$),
giving
\[
|\calNha(\phi,u-u_h)|\lesssim h\Norm{\phi}_{2,\Omega}\SemiNorm{u-u_h}_{1,h} \lesssim  h\Norm{u-u_h}_{0,\Omega}\SemiNorm{u-u_h}_{1,h}.
\]

To estimate the second term, 
\[
\begin{aligned}
  B(u-u_h,\phi_I) & =  B(u,\phi_I) -B(\uh,\phi_I)\\
  & = F(\phi_I) +\calNh(u,\phi_I)-F_h(\phi_I) +\Bh(\uh,\phi_I)-B(\uh,\phi_I)\\
  & = F(\phi_I)-F_h(\phi_I) +\calNh(u,\phi_I)+\Bh(\uh,\phi_I)-B(\uh,\phi_I)
\end{aligned}
\]

Next, thanks to the orthogonality and~(5.47) in~\cite{Beirao-Brezzi-Marini-Russo:2016}, we obtain
\[\begin{aligned}
  F(\phi_I)-F_h(\phi_I) = (f-\Pik f,\phi_I)= (f-\Pikh f,\phi_I-\Pikh \phi_I)&\lesssim h^{s+1}\SemiNorm{f}_{s,\Omega} \Norm{u-u_h}_{0,\Omega},\end{aligned}\]
  
Then a standard application of Lemma~\ref{lemma:nc-termall} together with the approximation property~\eqref{eqn:Approximation} of ~$\phi_I$ and the elliptic regularity theory:
  \[
  \calNh(u,\phi_I)= \calNh(u,\phi_I-\phi)\lesssim h^{s}\Norm{u}_{s+1,\Omega} \SemiNorm{\phi_I-\phi}_{1,h}\lesssim h^{s+1}\Norm{u}_{s+1,\Omega} \Norm{u-u_h}_{0,\Omega}.
\]

From Theorem \ref{lemma:ConsistencyL2},
we can get 
\[\Bh(\uh,\phi_I)-B(\uh,\phi_I) \lesssim ( h\SemiNorm{u-\uh}_{1,h}+h^{s+1}\Norm{u}_{s+1,\Omega})\Norm{u-u_h}_{0,\Omega}.\]

This completes the proof.
\end{proof}

\subsection{Error estimates in~$H^1$ and~$L^2$ norms}

With the results established in Lemmas~\ref{lemma:H1estimate} and~\ref{lemma:L2estimate}, we are ready to derive an error estimate for the virtual element method
approximation~$\uh$. To this end, we may substitute the result of Lemma~\ref{lemma:L2estimate} into the estimate shown in Lemma~\ref{lemma:H1estimate}.

\begin{theorem}\label{theorem:H1}
    In addition to the assumption of Lemma~\ref{lemma:L2estimate}, assume that the exact solution~$u$ is sufficiently smooth such that
~$u\in H^{s+1}(\Omega)$ with ~$1\leq s\leq k$. Then,
\begin{equation}\label{eqn:theH1estimate}
   \Norm{u-\uh}_{1,h} \lesssim  h^{s}(\Norm{u}_{s+1,\Omega}
+\SemiNorm{f}_{s-1,\Omega}),
\end{equation}
provided that the meshsize~$h$ is sufficiently small.
\end{theorem}

Now substituting the error estimate~\eqref{eqn:theH1estimate} into the estimate of Lemma~\ref{lemma:L2estimate}. The result can then be summarized as follows.
\begin{theorem}
    Under the assumption of Theorem~\ref{theorem:H1}, we have
\begin{equation}\label{eqn:theL2estimate}
   \Norm{u-\uh}_{0,\Omega} \lesssim  h^{s+1}(\Norm{u}_{s+1,\Omega}
+\SemiNorm{f}_{s-1,\Omega}),
\end{equation}
provided that the meshsize~$h$ is sufficiently small.
\end{theorem}

\section{The 3D version of the method}\label{section:3D}

Recalling the Section~$7$ in~\cite{Beirao-Liu-Mascotto-Russo:2023}, we have discussed the nonconforming virtual element method for the Poisson equation on~3D curved domain. 
In this section, we mainly presented the differences of the approach on~3D curved domain: the local nonconforming virtual element space and the definition of the nonconformity term. Firstly, the local nonconforming virtual element space on the (possibly curved) polyhedron $\E$ reads
\[
\begin{aligned}
 \VhE:=\{\vh \in H^1(\E):~&\Delta \vh \in \Pbb_{k}(\E),~
\nbfE \cdot \nabla \vh \in \Pbbtilde_{k-1}(F)\ 
~\forall F \in \EcalE, \\
&(\Pitilde_{k}^{\nabla,\E} v-v,q_k)_{\E}=0~\forall q_k\in (\Pbb_k(\E)/\Pbb_{k-2}(\E))\}, 
 \end{aligned}
\]
where $\Pbbtilde_{k-1}(F)$ is the push forward on $F$ of $\Pbb_{k-1}(\widehat{F})$. 

Next, the definition of the nonconformity term
which in 3D is defined as
\[
\calNh(u,v) 
:= \sum_{F\in\calEh^3} \int_F \llbracket v\rrbracket_F\cdot(\nabla u -\bbm u) \ {\mathrm{d}} F ,
\]
where~$\calEh^3$ denotes the set of (curved) faces in the polyhedral decomposition. 
Thus, the abstract error analysis is dealt with similarly to the 2D case.

\newpage

\section{Numerical experiments}\label{section:Numerical experiments}
In this section, we report numerical results for the nonconforming virtual element method on a variety of testing problems, with different mesh and domains.
To this end, let~$\uh$ and~$u$ be the solutions to the virtual element method equation ~\eqref{eqn:VEM} and the original equation~\eqref{eqn:primal problem}, respectively. 
Since the energy and~$L^2$ errors are not computable, we rather consider the computable error
quantities:
\begin{equation} \label{eqn:computable-quantities}
\begin{aligned}
E_{H^1}
:= \frac{\left(\sum_{\E\in\calTh}
\SemiNorm{  u-  \Pitilde_{k}^{0,\E}\uh}_{1,\E}^2\right)^{\frac12}}{\Norm{u}_{1,\Omega}},
\qquad 
E_{L^2}
:= \frac{\left(\sum_{\E\in\calTh}
\Norm{u- \Pitilde_{k}^{0,\E} \uh}_{0,\E}^2\right)^{\frac12}}{\Norm{u}_{0,\Omega}}.\\
\end{aligned}
\end{equation}
\subsection{Case 1: model problem on the unit square}
First, we consider the  variable coefficients of the equation are given by
\begin{equation}\label{coefficients}
    a_1(x,y)= \left( \begin{aligned}
    y^2+1&&-xy\\
    -xy&&x^2+1
\end{aligned}\right),~\bbm_1(x,y) = (x,y),~c_1(x,y) = x^2+y^3,
\end{equation}
and Dirichlet boundary conditions defined in such a way
that the exact solution is
\[u_1(x,y) = x^2y+\sin(2\pi x)\sin(2\pi y)+2.\]
\begin{figure}[H]
\label{fig:squaremesh}
    \centering
\includegraphics[width=2.5in]{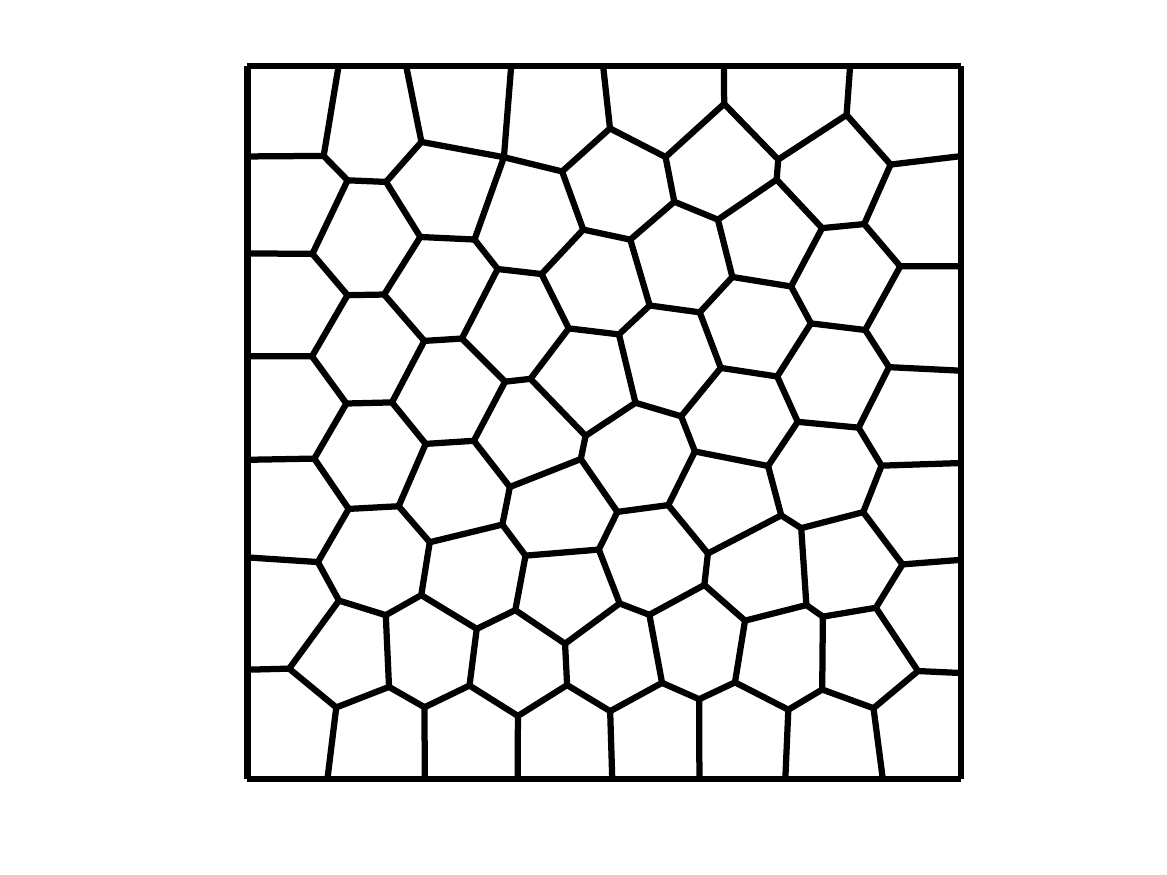}\includegraphics[width=2.5in]{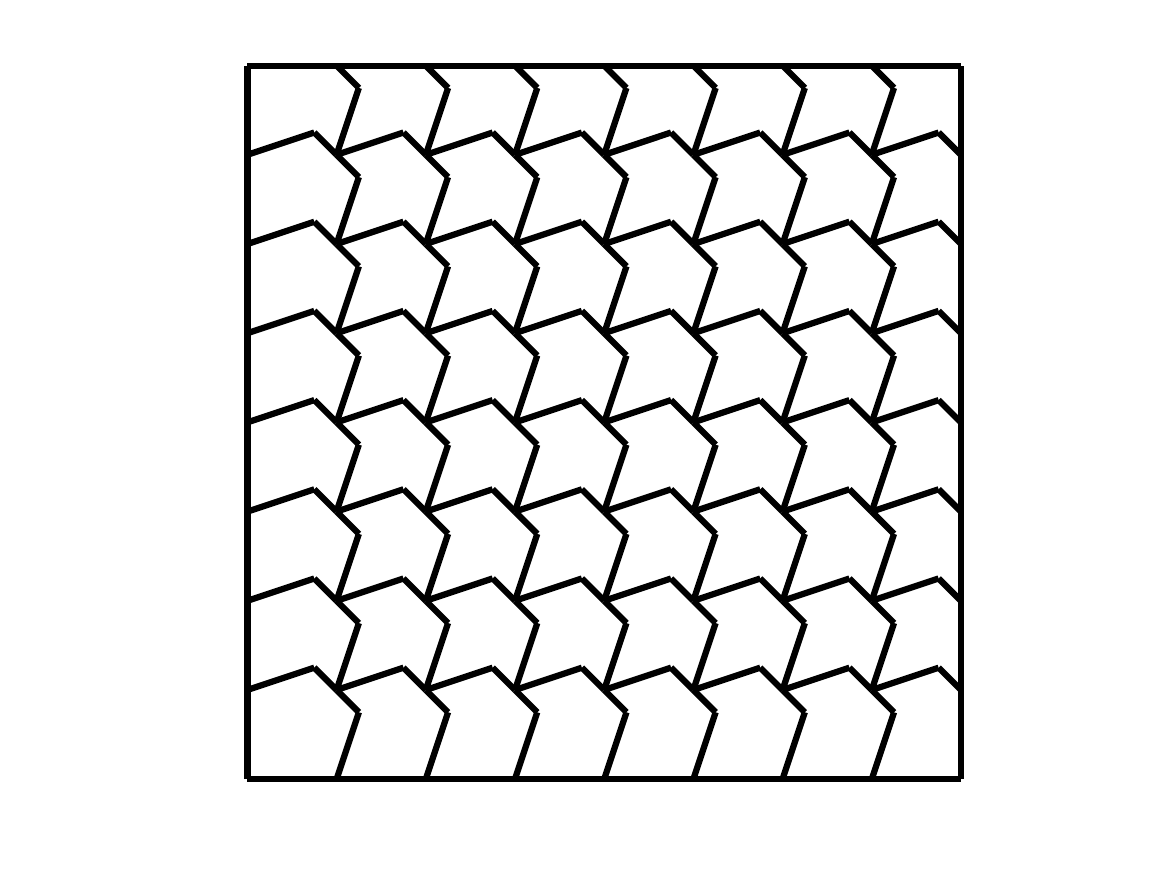}
\caption{Left-panel: an example of Voronoi mesh over~$\Omega$. Right-panel: an example of concave mesh over~$\Omega$.}
\end{figure}

In Figure~\ref{figure:convergence-test-case-1} ,
we show the convergence of the two error quantities in~\eqref{eqn:computable-quantities} on the given sequence of Voronoi meshes and concave meshes with decreasing mesh size;
see Figure~\ref{fig:squaremesh} for sample meshes.
We consider virtual elements of ``orders'' $k=1$,~$2$,~$3$, and~$4$.

\begin{figure}[H]

\centering
\includegraphics[width=2.5in]{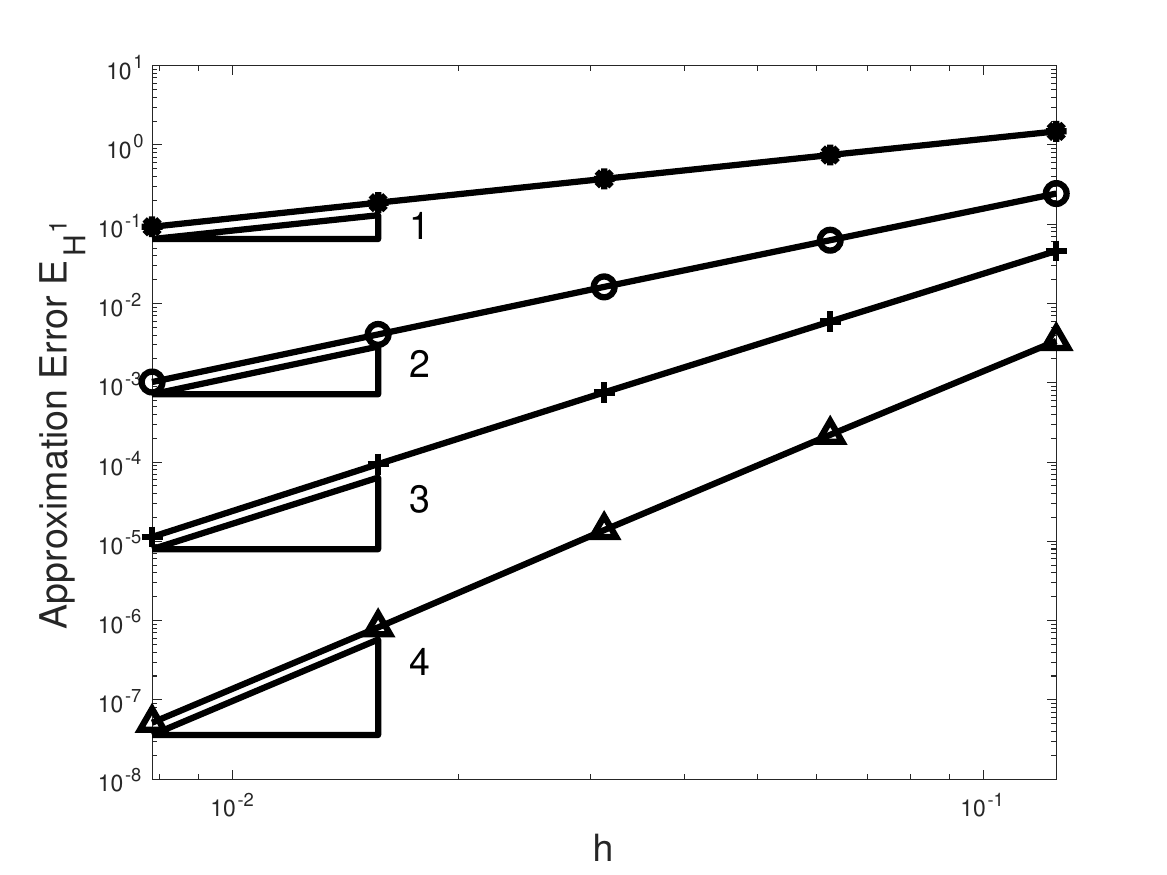}\includegraphics[width=2.5in]{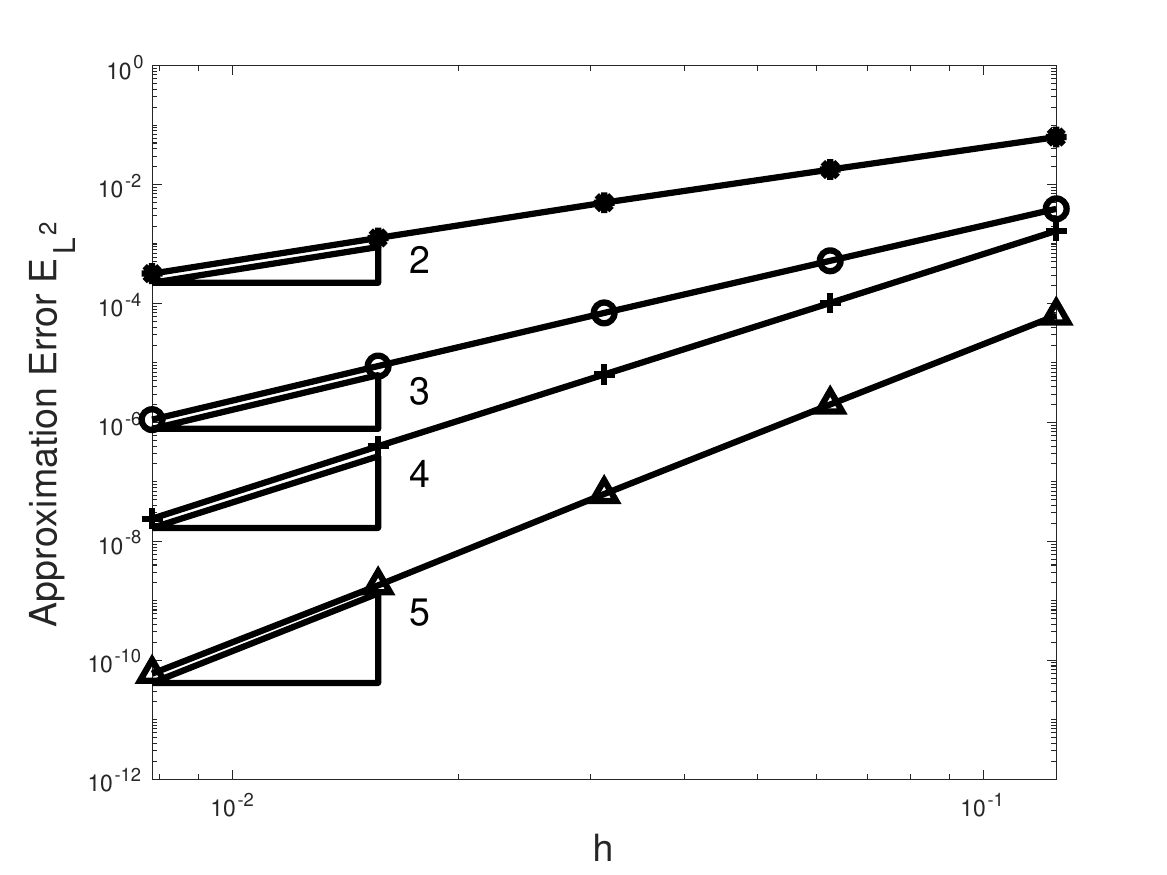}

\includegraphics[width=2.5in]{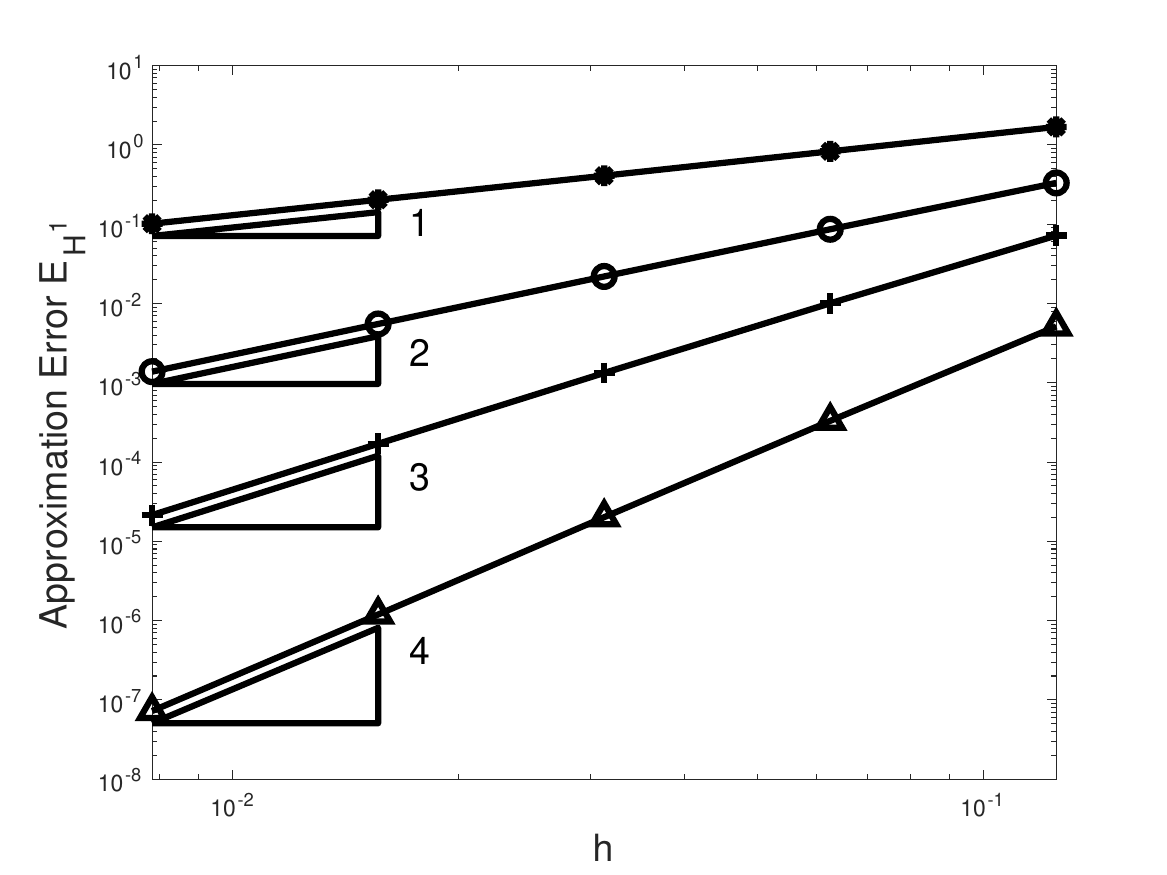}\includegraphics[width=2.5in]{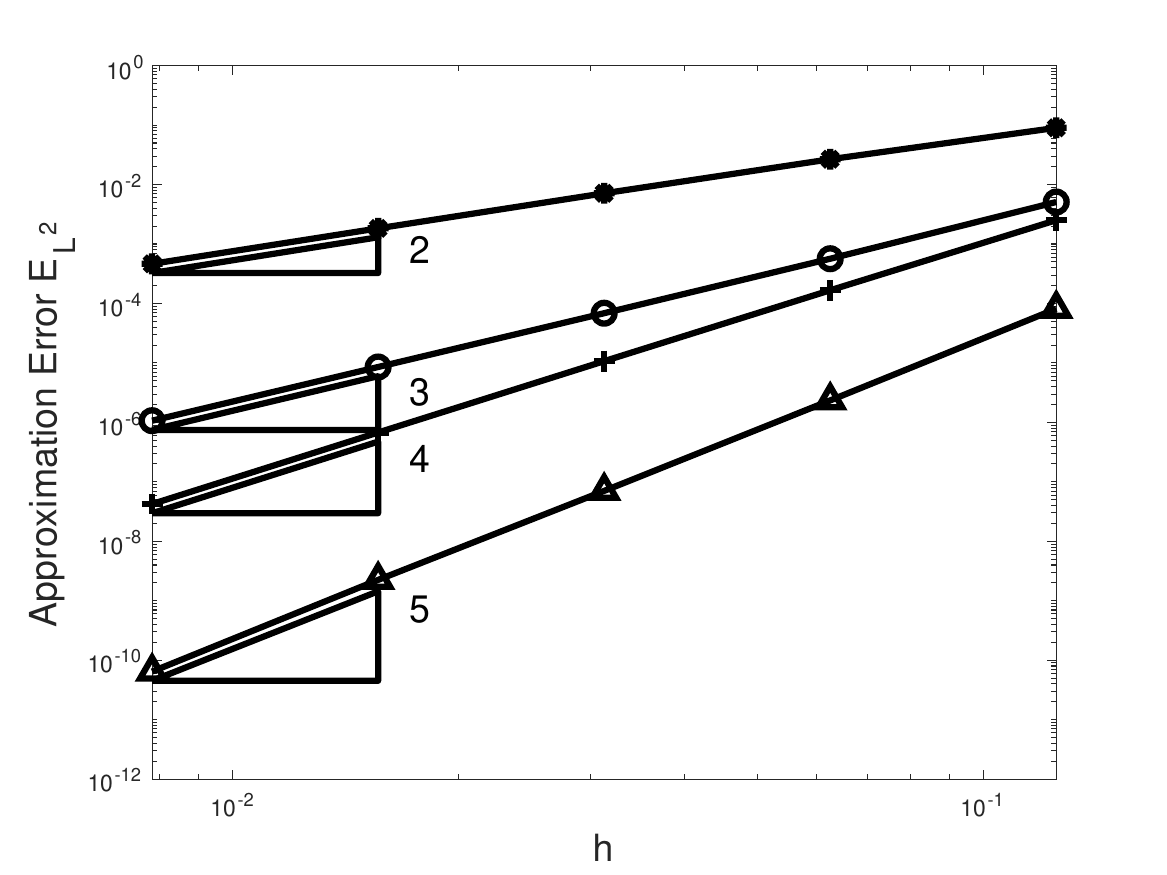}
    \caption{Left-panel: the convergence of~$E_{H^1}$. Right-panel: the convergence of~$E_{L^2}$. The exact solution is ~$u_1$.
We employ sequences of Voronoi meshes (first row)
and concave meshes (second row)
with decreasing mesh size are employed. The “orders” of the virtual element spaces are
~$k =1$,~$2$,~$3$, and~$4$.}
\label{figure:convergence-test-case-1}
\end{figure}
\subsection{Case 2: model problem on curved quadrilateral domain}

We consider  variable coefficients given by~\eqref{coefficients} of the other equation on a curved quadrilateral domain presented in~\cite{BeiraodaVeiga-Russo-Vacca:2019}
and defined as
\begin{equation} \label{eqn:sindomain}
  \Omega:=\{(x, y) ~\text{s.t}~ 0 < x < 1,~\text{and}~ g_1(x) < y < g_2(x)\},    
\end{equation}
where
\[
g_1(x):= \frac{1}{20}\sin(\pi x)\quad \text{and}\quad g_2(x):= 1 + \frac{1}{20}\sin(3\pi x).
\]

We represent the domain in Figure~\ref{figure:test-case-2domain}.
On such an~$\Omega$, we consider
the exact (analytic) solution
\[
u_2(x,y) =-(y-g_1(x))(y-g_2(x))(1-x)x(3+\sin(5x)\sin(7y)).
\]
The function~$u_2$ has homogeneous Dirichlet boundary conditions over~$\partial \Omega$.

\begin{figure}[H]
\begin{center}
\includegraphics[width=2.5in]{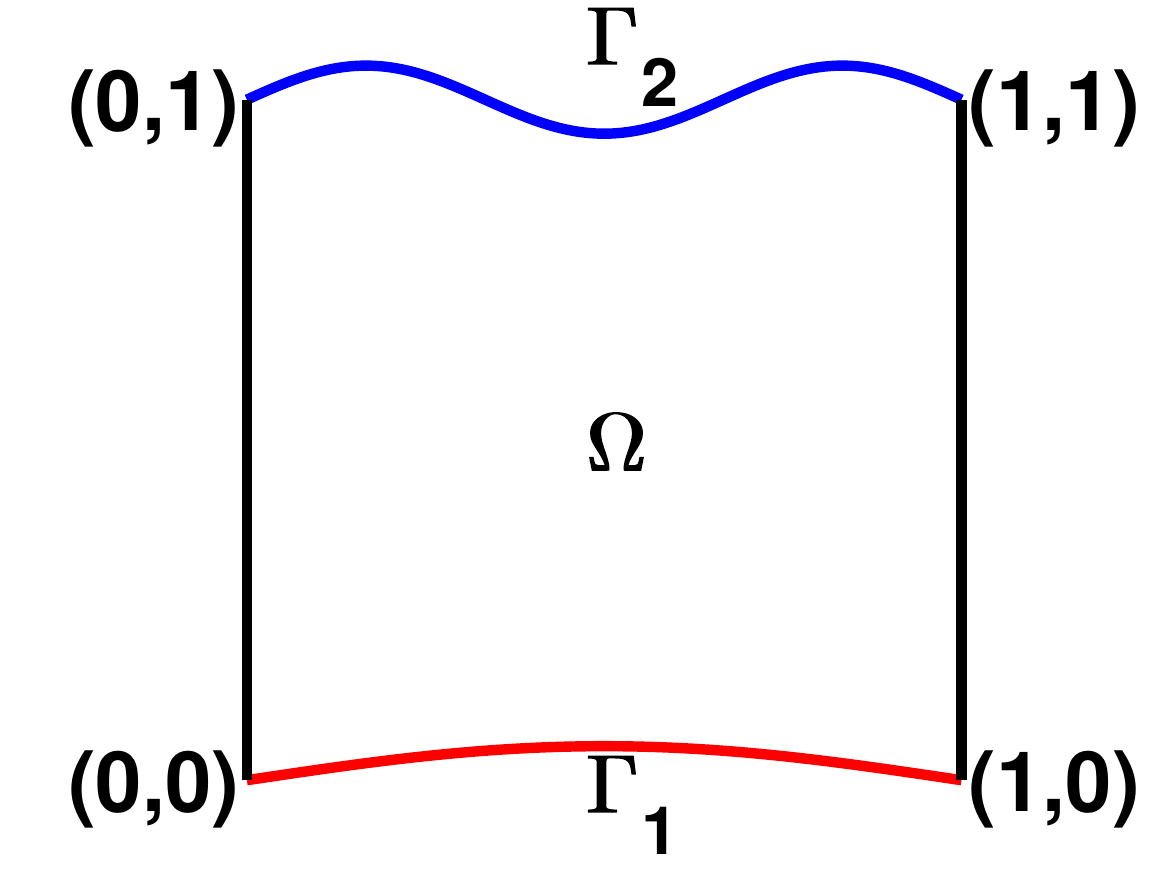}
\caption{Domain~$\Omega$ described in \eqref{eqn:sindomain}.}
\label{figure:test-case-2domain}
\end{center}
\end{figure}
\begin{figure}[H]
\begin{center}
\includegraphics[width=2.5in]{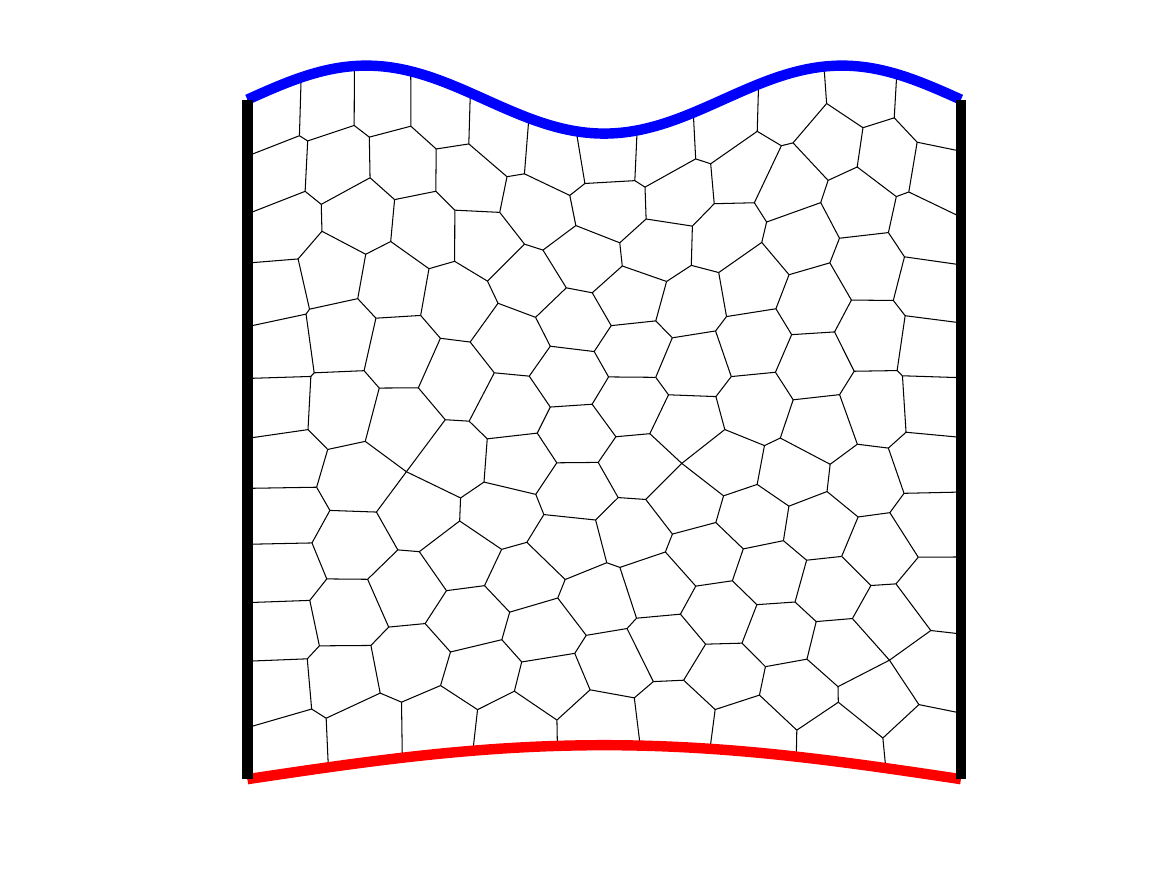}
\includegraphics[width=2.5in]{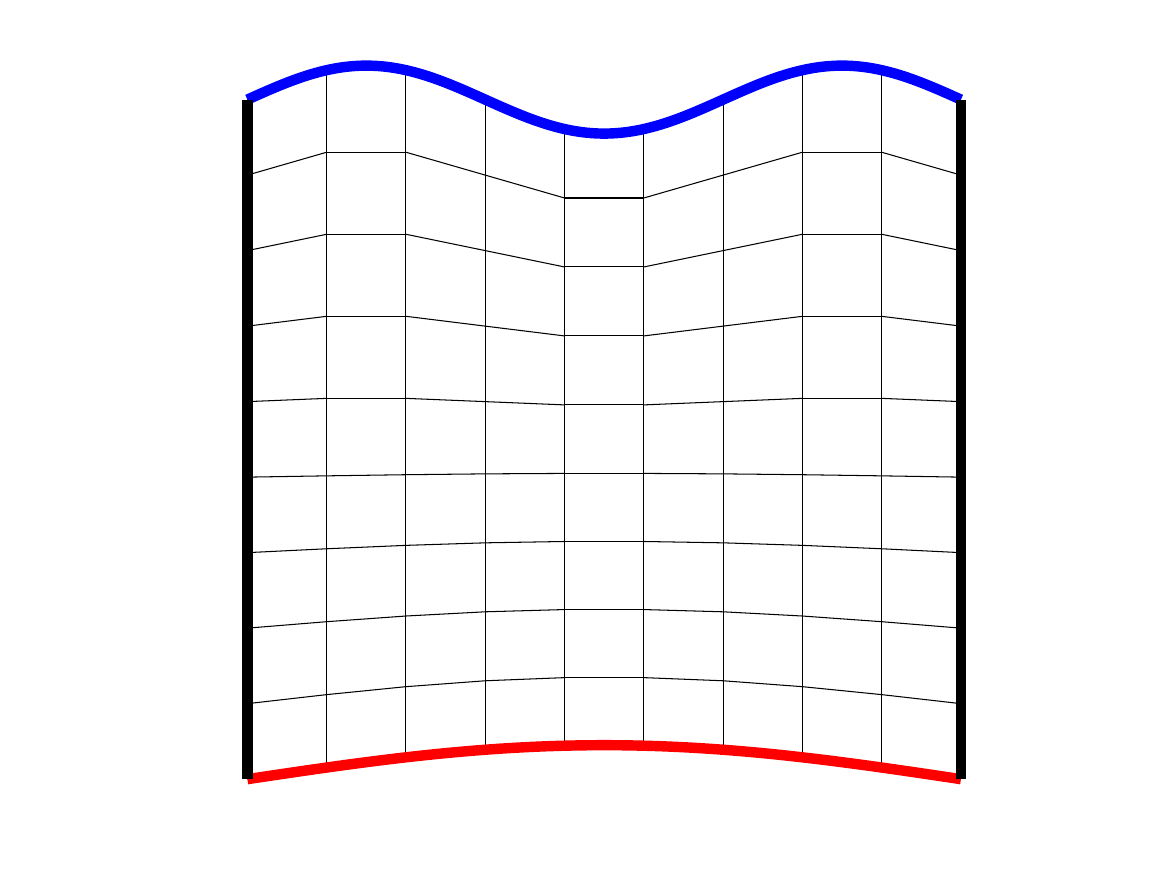}
\caption{Left-panel: an example of (curved) Voronoi mesh over~$\Omega$.
Right-panel: an example of (curved) quadrilateral mesh over~$\Omega$.}
\label{figure:test-case-2}
\end{center}
\end{figure}

The finite element partition on the curved domain~$\Omega$
is constructed starting from a mesh for the square $(0,1)^2$ and mapping the nodes accordingly to the following rule:

\[
(x_{\Omega},y_\Omega)=
\begin{cases}
(x_s,y_s+g_1(x_s)(1-2y_s)),   &\text{if}~ y_s \leq \frac12 ,\\
(x_s, 1-y_s+g_2(x_s)(2y_s-1), & \text{if}~ y_s \geq \frac12.
\end{cases}
\]
Above, $(x_s, y_s)$
denotes the mesh generic node on the square domain~$(0,1)^2$, and~$(x_\Omega,y_\Omega)$
denotes the associated node in the curved domain~$\Omega$.
The edges on the curved boundary consist of
an arc of~$\Gamma_1$ or~$\Gamma_2$, while all the internal edges are straight. 
In Figure~\ref{figure:test-case-2},
we display two examples of meshes,
namely, a (curved) Voronoi and a (curved) square mesh.

In Figure~\ref{figure:convergence-test-case-2},
we show the convergence of the two error quantities in~\eqref{eqn:computable-quantities}
on the given sequences of meshes under uniform mesh refinements
for ``orders'' $k=1$, $2$, $3$, and~$4$.

\begin{figure}[H]
\begin{center}
\includegraphics[width=2.5in]{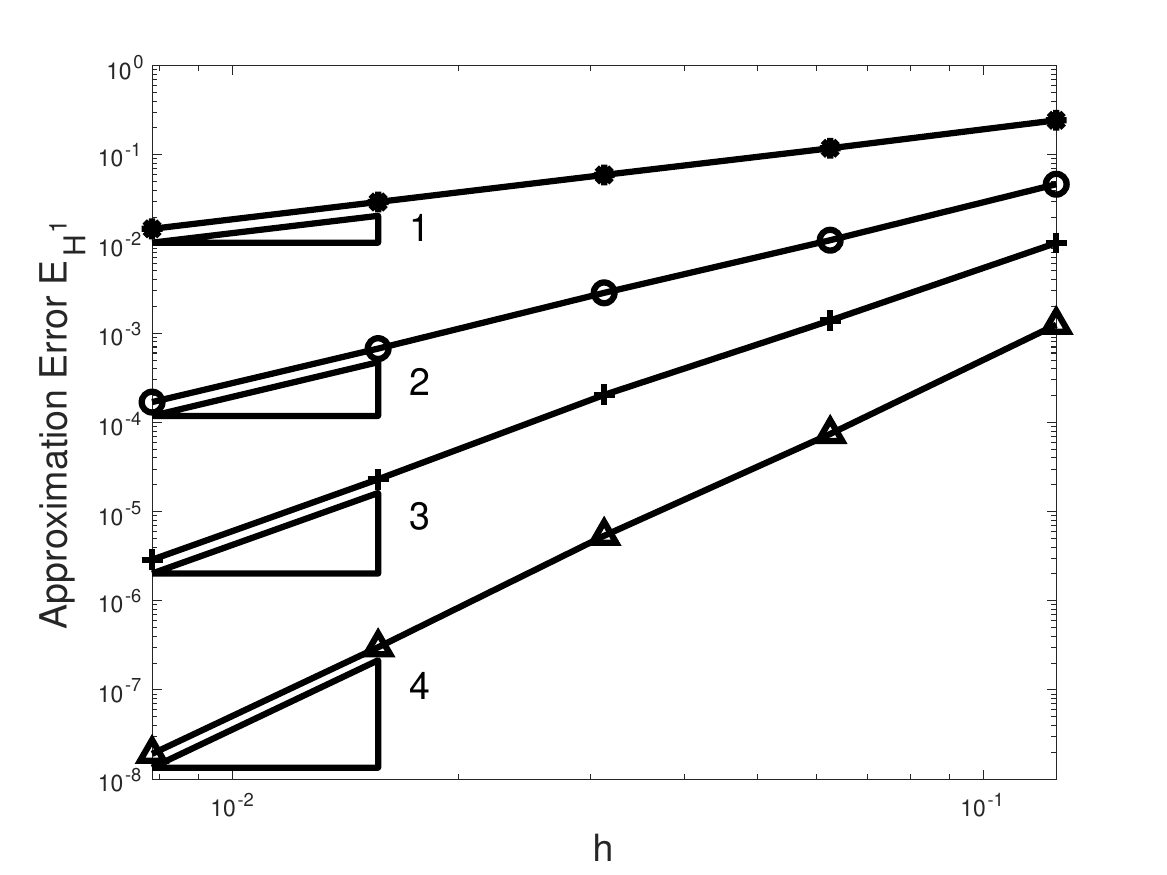} 
\includegraphics[width=2.5in]{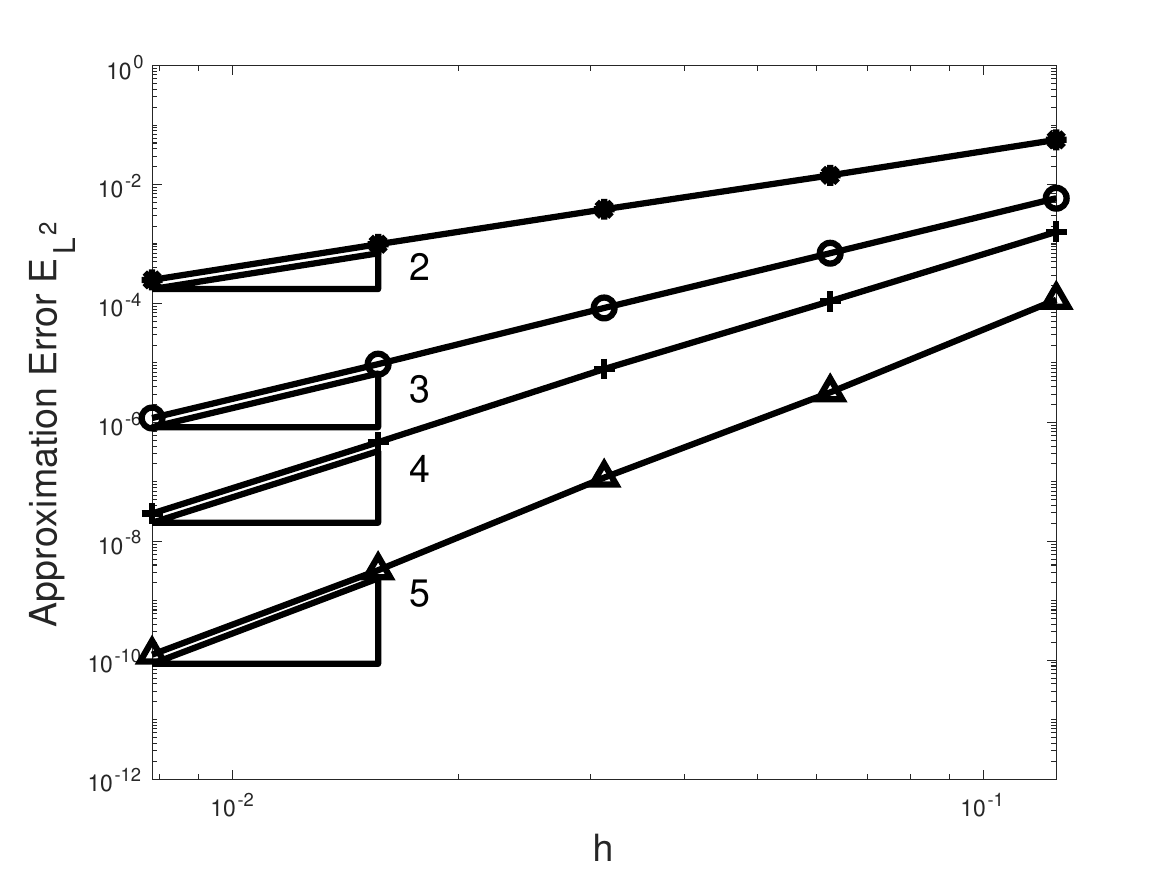}
\includegraphics[width=2.5in]{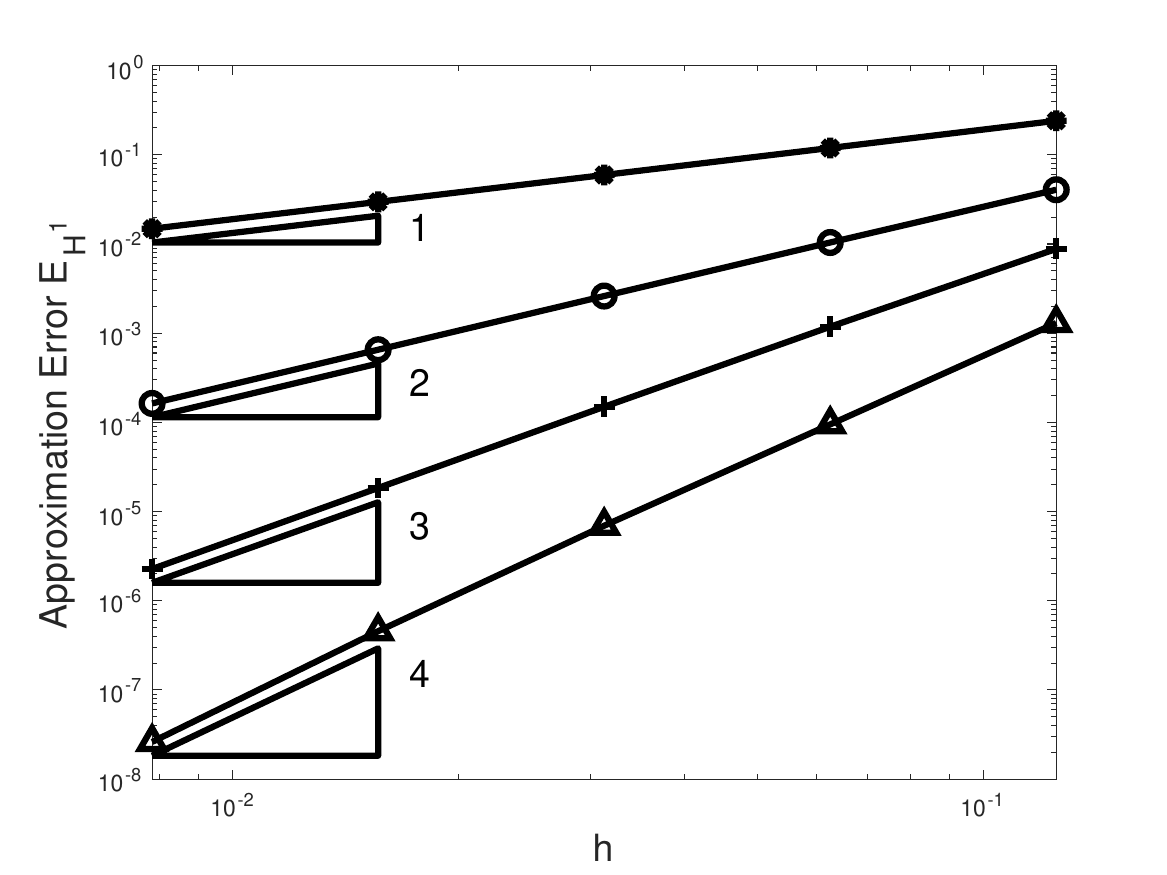} 
\includegraphics[width=2.5in]{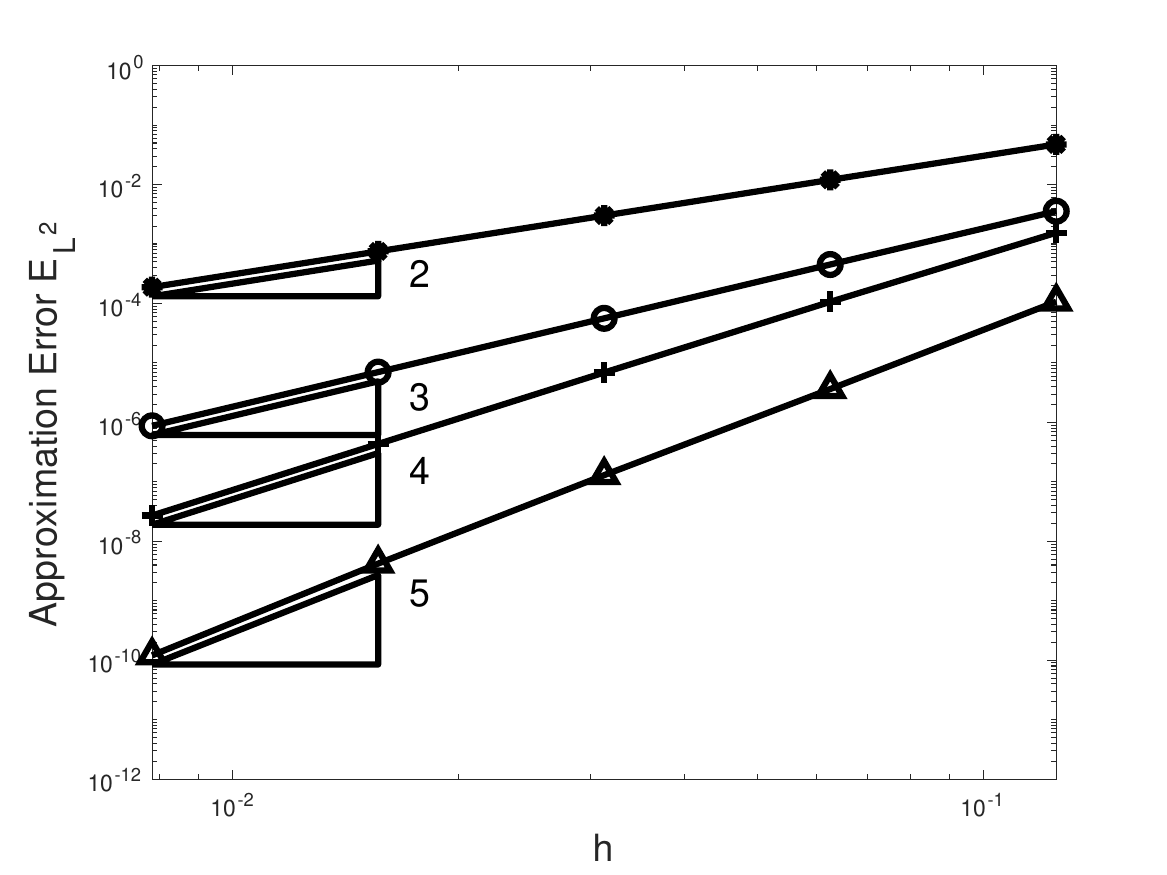}
\caption{\small{Left-panels: the convergence of~$E_{H^1}$.
Right-panels: the convergence of~$E_{L^2}$.
The exact solution is~$u_2$.
We employ sequences of (curved)} Voronoi meshes (first row)
and (curved) square meshes (second row)
with decreasing mesh size are employed.
The ``orders'' of the virtual element spaces are
$k=1$, $2$, $3$, and~$4$.}
\label{figure:convergence-test-case-2}
\end{center}
\end{figure}

\subsection{Case 3: model problem on domian with curved interface}
Here, we consider the general elliptic equation defined on a square~$(0,1)\times(-\frac12,\frac12)$ with  interface~$\Gamma$ given as~
$g_3(x): = \frac{1}{20}\sin(3\pi x)$ (Figure \ref{figure:test-case-3}).
 And then, the exact solution is designed to be
\[
 u_3 =\left\{\begin{aligned}
    &\frac{x(1-x)(y-g_3(x))(3+\sin(5x)\sin(7y))}{\kappa_1},\quad\forall x\in\Omega_1,\\
    &\frac{x(1-x)(y-g_3(x))(3+\sin(5x)\sin(7y))}{\kappa_2},\quad\forall x\in\Omega_2.
 \end{aligned}\right.
 \]
with the coefficients~$\bbm_3 = \bbm_1$,~$c_3 = c_1$~ and 
\[
 a_3(x,y)= \left\{ \begin{aligned}
    \kappa_1 a_1(x,y),\quad\forall x\in\Omega_1,\\
    \kappa_2 a_1(x,y),\quad\forall x\in\Omega_2.
\end{aligned}\right.
\]
\begin{figure}[H]
\begin{center}
\includegraphics[width=2.5in]{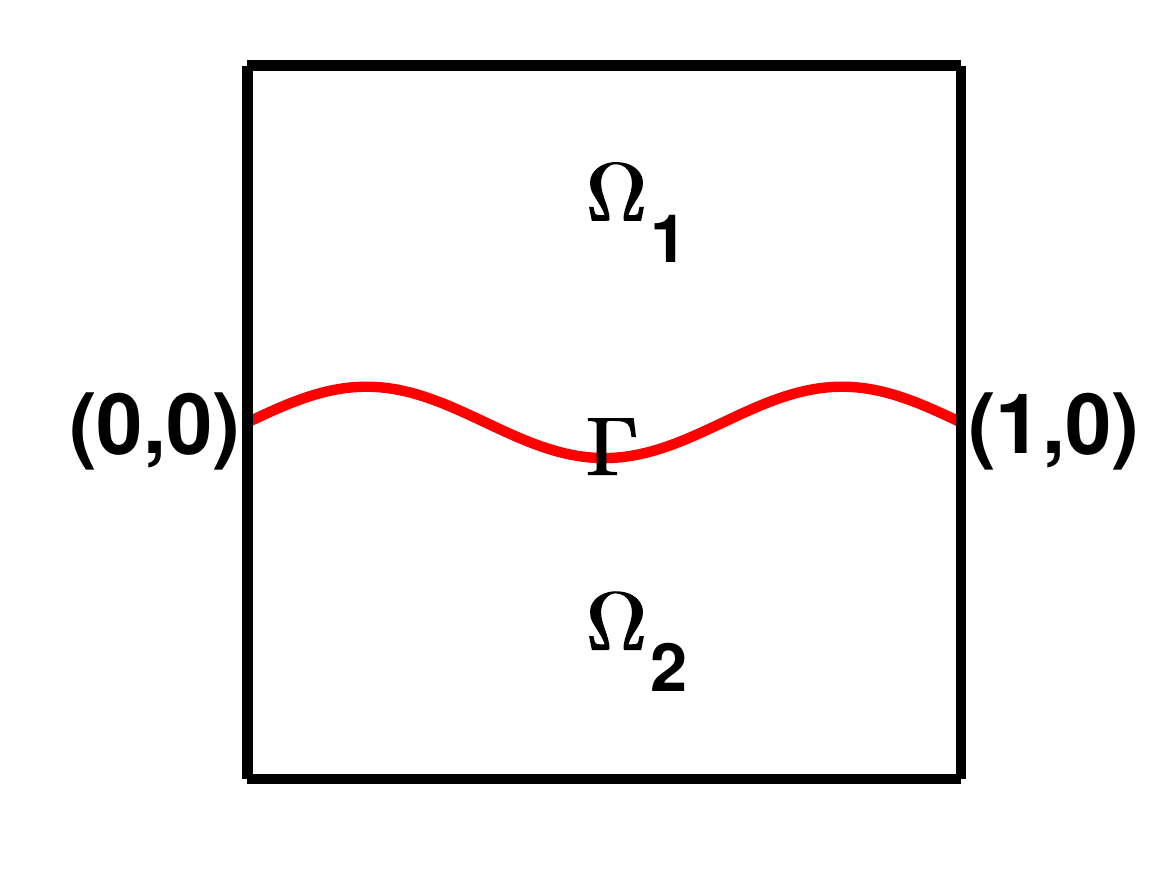}
\includegraphics[width=2.5in]{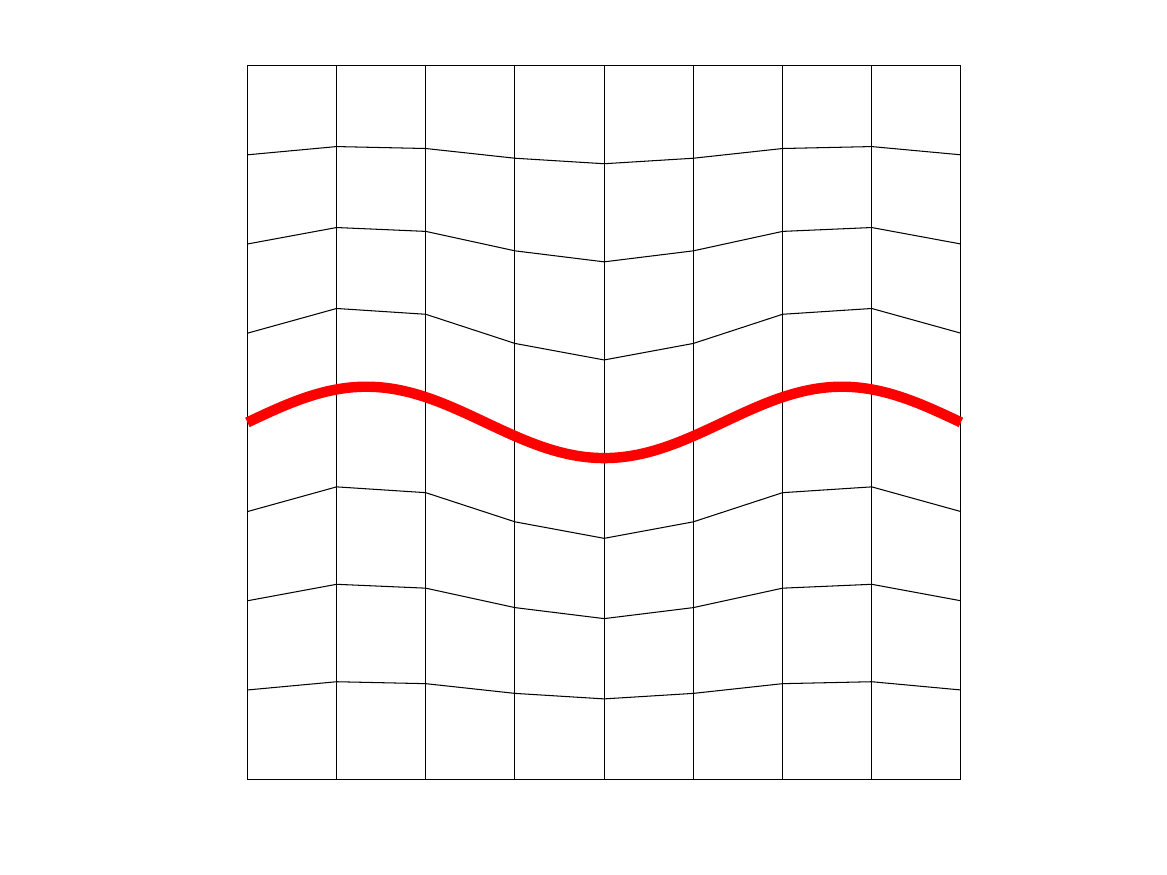}
\caption{Left-panel: the square with curved interface ~$\Gamma$.
Right-panel: an example of (curved) quadrilateral mesh over~$\Omega$.}
\label{figure:test-case-3}
\end{center}
\end{figure}

In Figure~\ref{figure:convergence-test-case-3},
we show the convergence of the two error quantities in~\eqref{eqn:computable-quantities}
on the given sequences of meshes under uniform mesh refinements
for ``orders'' $k=1$, $2$, $3$, and~$4$.

\begin{figure}[H]
\begin{center}
\includegraphics[width=2.5in]{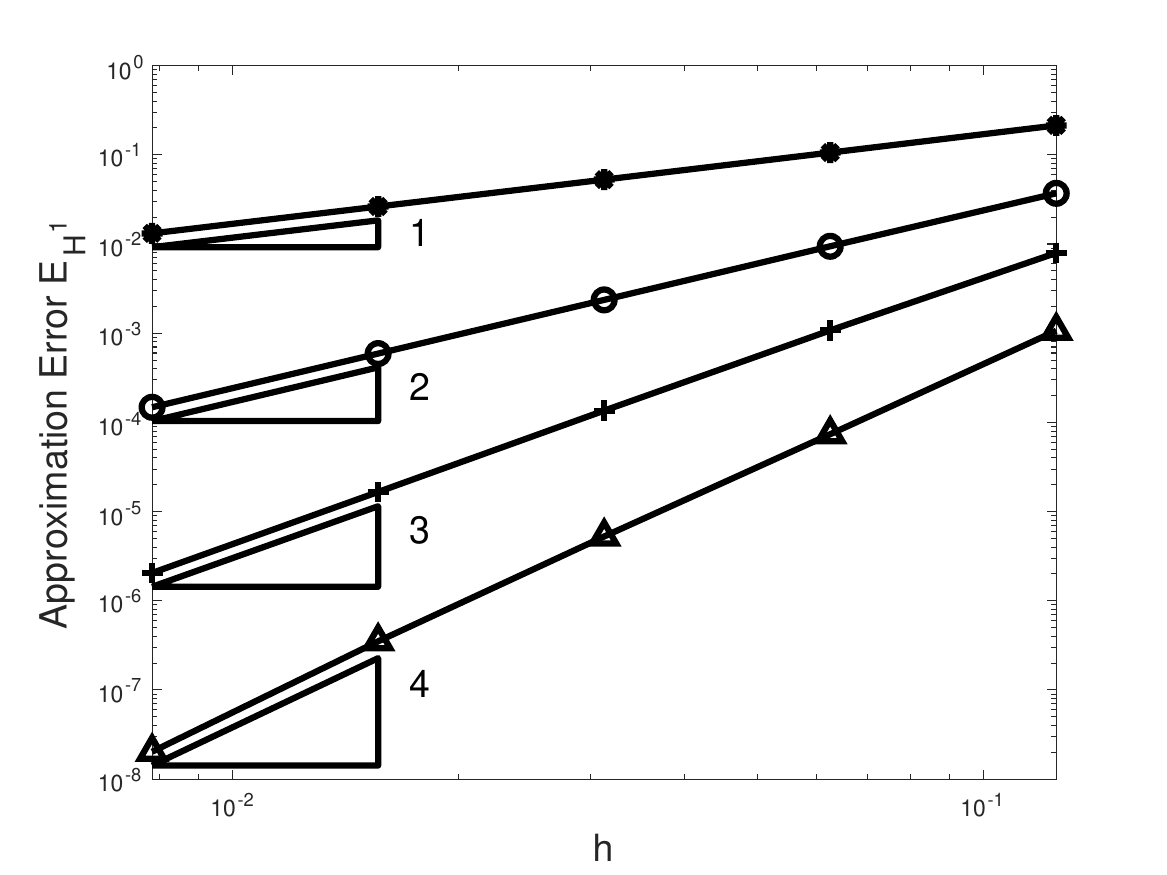} 
\includegraphics[width=2.5in]{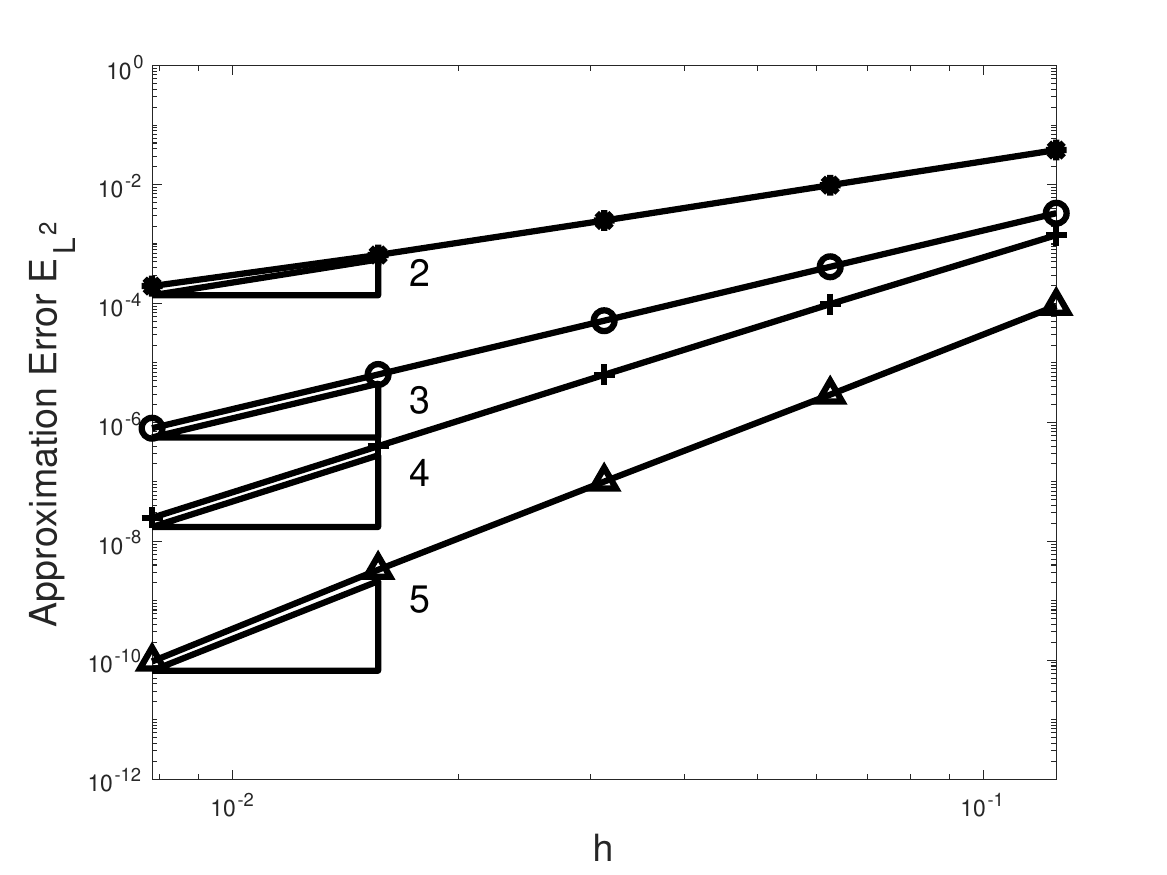}
\includegraphics[width=2.5in]{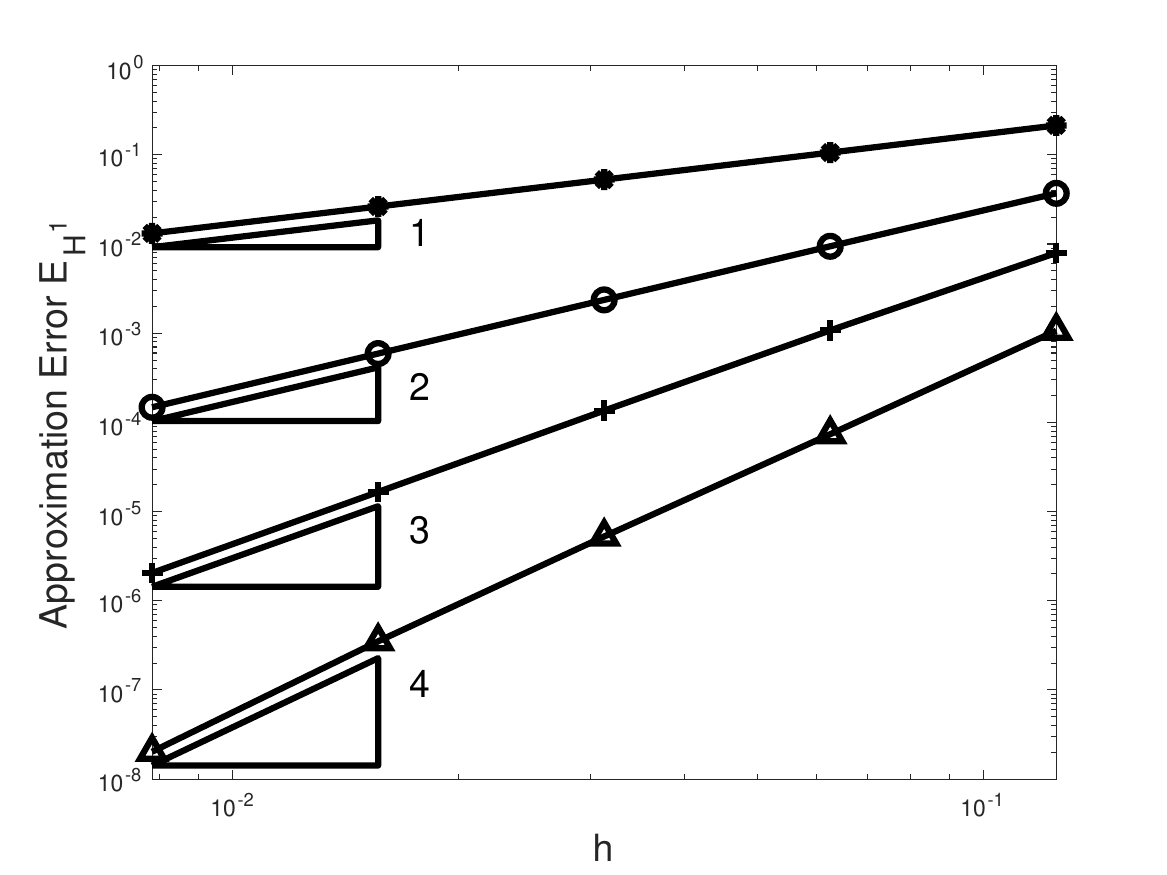} 
\includegraphics[width=2.5in]{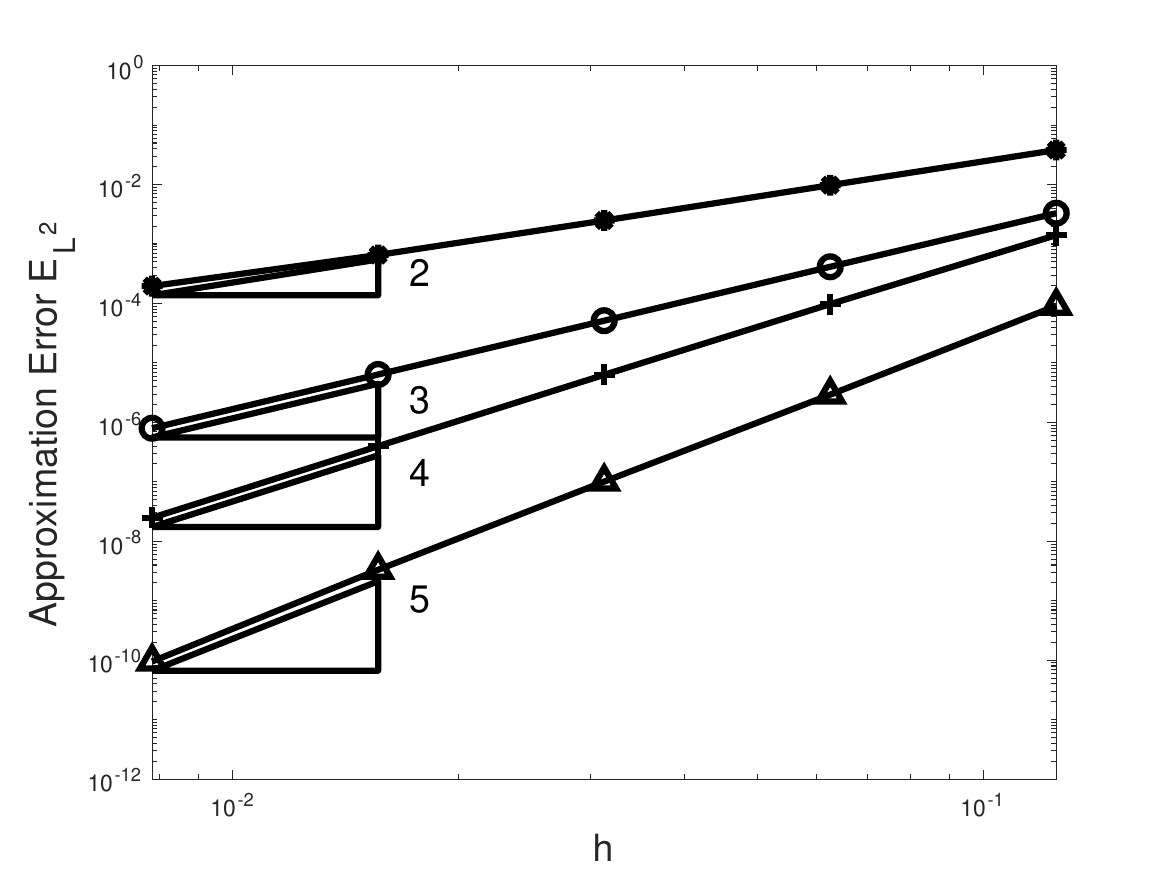}
\caption{Left-panels: the convergence of~$E_{H^1}$.
Right-panels: the convergence of~$E_{L^2}$.
The exact solution is~$u_2$.
We employ the coefficients~$\kappa_1 = 1$ and $\kappa_2 = 10^5$ (first row)
and ~$\kappa_1 = 10^5$ and $\kappa_2 = 1$ (second row)
on the quare meshes with decreasing mesh size are employed.
The ``orders'' of the virtual element spaces are
$k=1$, $2$, $3$, and~$4$.}
\label{figure:convergence-test-case-3}

\end{center}
\end{figure}
All numerical results are in excellent agreement with the theory established in this paper.
\bibliographystyle{plain}
{\footnotesize\bibliography{bibliography.bib}}

\end{document}